\documentclass[invmat,francais]{svjour}

\def\G_m{{\bf G}_m}

\def\e{{\varepsilon}}

\let\G\Gamma

\let\w\omega

\def\char{\mathop{\rm char}\nolimits}

\def\Gal{\mathop{\rm Gal}\nolimits}

\def\Hom{\mathop{\rm Hom}\nolimits}

\def\id{\mathop{\rm id}\nolimits}
\def\ind{\mathop{\rm ind}\nolimits}

\def\id{\mathop{\rm id} \nolimits}

\def\lim{\mathop{\rm lim} \nolimits}

\def\inf{\mathop{\rm inf} \nolimits}

\def\Gal{\mathop{\rm Gal}\nolimits}
\def\choosenot#1#2{\setbox0=\hbox{$\mathsurround0pt#1#2$}\rlap{\hbox
to\wd0{\hfil$\mathsurround0pt#1/$\hfil}}\hbox{$\mathsurround0pt#1#2$}}

\def\not#1{\mathchoice{\choosenot{\displaystyle}{#1}}{\choosenot{\textstyle}{#1}}{\choosenot{\scriptstyle}{#1}}{\choosenot{\scriptscriptstyle}{#1}}}
\begin{document}

\centerline {\bf De $GL(2,F)$ \`a  $\Gal_{\bf Q_{p}}$}

\bigskip 
\centerline{Marie-France Vigneras}

\bigskip
\centerline{Paris Juillet 2009, r\'evis\'e septembre 2009}

% $\stackrel \sim \to$
%$\mathop{\to}\limits ?? \sim$

\section{Introduction} Soit $F$ une extension finie de $\bf Q_{p}$ de corps r\'esiduel  un corps fini $\bf F_{q} $ \`a $q$ \'el\'ements, ou $F= {\bf F}_{q} ((T))$.   Soit $L$ une extension finie de $\bf Q_{p}$, d'anneau des entiers $o$ et de corps r\'esiduel $k$. 
Notons  $\Gal_{F}$ le groupe de Galois absolu de $F$ et   ${\cal M}_{o}(G)$ la cat\'egorie ab\'elienne des repr\'esentations lisses de  $G:=GL(2,F)$ sur des $o$-modules. 

En utilisant la th\'eorie des $(\varphi, \Gamma)$-modules, 
Colmez a  d\'efini un foncteur  exact  contravariant de la cat\'egorie des repr\'esentations lisses de $GL(2, {\bf Q_{p}})$ sur $o$ de longueur finie ayant un caract\`ere central (donc admissibles) dans la cat\'egorie ${\cal M}_{o}^{f}(\Gal _{\bf Q_{p}})$  des  repr\'esentations  continues de $\Gal_{\bf Q_{p}}$ dans des $o$-modules finis, dont il a d\'eduit la correspondance de Langlands par d\'eformation pour $\Gal_{\bf Q_{p}}$ et $GL(2, {\bf Q_{p}})$.  

Une g\'en\'eralisation de ce foncteur a \'et\'e donn\'ee lorsque l'on remplace $G$ par un groupe r\'eductif $p$-adique d\'eploy\'e sur $\bf Q_{p}$ de centre connexe dans \cite{SVig}.

 Ici, nous ne supposons pas $ \bf Q_{p}=F$ et nous allons construire un foncteur  d'une sous-cat\'egorie de   ${\cal M}_{o}(G)$ vers la cat\'egorie des  $(\varphi, \Gamma)$-modules.

On dit que $V\in {\cal M}_{o}(G)$ admet une pr\'esentation finie (relative \`a $KZ$), s'il existe
 une suite exacte dans  la sous-cat\'egorie ${\cal M}_{o}^{tf}(G)$ des repr\'esentations de type fini dans ${\cal M}_{o}(G)$,
$$0 \to R_{V}(V_{0}) \to \ind_{KZ}^{G}(V_{0}) \to V \to 0 \quad , $$
o\`u  $V_{0}\subset V|_{KZ}$, $Z$ le centre de $G$, $K=GL(2,O_{F})$ et $\ind_{KZ}^{G}$ est le foncteur d'induction compacte de $KZ$ \`a $G$. Notons $ {\cal M}_{o-tor}(G)$ la sous-cat\'egorie des repr\'esentations de $ {\cal M}_{o}(G)$ sur des $o$-modules de torsion et 
$$ {\cal M}_{o-tor}^{apf}(G) \quad,\quad  {\rm resp.} \quad  {\cal M}_{o-tor}^{alf}(G) \quad, 
$$
les sous-cat\'egories   des repr\'esentations  de $G$ admissibles de pr\'esentation finie, resp.  admissible s de longueur finie, dans $ {\cal M}_{o-tor}(G)$. 

Enon\c{c}ons  le th\'eor\`eme principal.

\begin{theorem}\label{0}  1) Une repr\'esentation dans $ {\cal M}_{o-tor}(G)$  est admissible et de pr\'esentation finie, si et seulement si elle est de longueur finie et ses sous-quotients irr\'eductibles sont admissibles et de pr\'esentation finie.

2)  Supposons que $F$ soit  une extension finie de $\bf Q_{p}$.  On peut d\'efinir un foncteur contravariant (th. \ref{th3}, cor. \ref{cor5})
$${\cal V} \quad : \quad {\cal M}_{o-tor}^{apf}(G) \quad \to \quad {\cal M}_{o}^{f}(\Gal _{\bf Q_{p}}) \quad . $$

3) Lorsque $F={\bf Q_{p}}$ les cat\'egories $  {\cal M}_{o-tor}^{apf}(GL(2,{\bf Q_{p}})) , 
 {\cal M}_{o-tor}^{alf}(GL(2,{\bf Q_{p}}))$ sont \'egales, le foncteur 
   contravariant $${\cal V}\quad : \quad {\cal M}_{o-tor}^{alf}(GL(2,{\bf Q_{p}})) \quad \to \quad {\cal M}_{o}^{f}(\Gal _{\bf Q_{p}})$$
est exact, et coincide sur les repr\'esentations ayant un caract\`ere central, avec le foncteur de Colmez.
 \end{theorem}

L'int\^eret de ce r\'esultat d\'epend de l'existence de   $o$-repr\'esentations irr\'eductibles de $G$ admissibles  de pr\'esentation finie lorsque $F\neq \bf Q_{p}$. On sait que les repr\'esentations irr\'eductibles non supersinguli\`eres  sont admissibles de pr\'esentation finie. Malheureusement, on n'a semble-t-il, aucun exemple - ou contre-exemple -  de cette propri\'et\'e dans le cas supersingulier.

\bigskip La d\'emonstration du th\'eor\`eme principal repose sur un th\'eor\`eme technique  dans la cat\'egorie $ {\cal M}_{o}(B)$ o\`u $B$ est le sous-groupe triangulaire sup\'erieur  de $G$. 
 
\bigskip  Pour \'enoncer le th\'eor\`eme technique, introduisons des notations suppl\'ementaires. Notons $B=TU$ la d\'ecomposition de Levi usuelle, $B_{0}=T_{0}U_{0}$ le sous-groupe ouvert compact des points $O_{F}$-entiers, $ t :=\pmatrix {p_{F}&0 \cr 0&1} \  ,$ le monoide  $$B^{+}\quad:=\quad B_{0}Zt^{\bf N}\quad  = \quad \cup_{n \in \bf N}\ B_{0}Zt^{n} \quad , $$  $(B^{+})^{-1}$ son inverse. 
 On identifie naturellement $U(q):=U(O_{F}/p_{F}O_{F})$ \`a un syst\`eme de repr\'esentants de $U_{0}/tU_{0}t^{-1}$.  
On dira que $D\in  {\cal M}_{o}((B^{+})^{-1})$ est \'etale si l'application $$d\mapsto ((ut)^{-1}d)_{u\in U(q)}\quad : \quad D\quad \to \quad \oplus_{u\in U(q)}D$$ est bijective. La cat\'egorie  ${\cal M}_{o}^{et}((B^{+})^{-1})$  des $o$-repr\'esentations lisses \'etales de $(B^{+})^{-1}$ est ab\'elienne.

\bigskip 

Soit $V\in {\cal M}_{o}(B)$ admettant une pr\'esentation finie (relative \`a $B_{0}Z$). On choisit, comme on le peut,  une suite exacte dans  ${\cal M}_{o}^{tf}(B)$,  
$$0 \to R_{V}(V_{0}) \to \ind_{B_{0}Z}^{B}(V_{0}) \to V \to 0 \quad , $$
avec  $V_{0}\subset V|_{B_{0}Z}$. 
Soit $$0\ \to \ V' \ \to \ V  \ \to \  V''\  \to \ 0$$ une suite exacte dans ${\cal M}_{o}(B)$ de repr\'esentations de pr\'esentations finies. On choisit, comme on le peut,  des $V'_{0},V_{0}, V''_{0}$ comme ci-dessus dans les restrictions \`a $B_{0}Z$ de $V',V,V''$ formant 
 une suite exacte   $$0\ \to \ V'_{0}\  \to \ V_{0} \ \to \ V_{0}'' \ \to \ 0$$ dans ${\cal M}_{o}^{tf}(B_{0}Z)$ 
 et induisant des pr\'esentations finies de $V',V,V''$.

\begin{theorem}\label{tech} Modulo $ {\cal M}^{ft}_{o}(B_{0}Z)$,

a)  $<B^{+}V_{0}> $ ne d\'epend pas du choix de $V_{0}$,

b) $\Delta_{V}(V_{0})\ := \ <B^{+}V_{0}> \cap  <(B-B^{+})V_{0}> \ \equiv \ 0 $,

c) $$D_{V}(V_{0}) \ := \ {<B^{+}V_{0}> \over \Delta_{V}(V_{0})} \ \in \ {\cal M}_{o}^{et}((B^{+})^{-1})\quad , $$ 

d) $0\ \to  \  <B^{+}V'_{0}> \ \to \   <B^{+}V_{0}>\ \to \  <B^{+}V''_{0}>\ \to \ 0$ est une suite exacte,

e) si $V$ se prolonge en une repr\'esentation admissible de $G$, alors $<B^{+}V_{0}>^{U_{0}} \  \equiv \ 0.$
\end{theorem}

On a not\'e $<B^{+}V_{0}> $ et $<(B-B^{+})V_{0}> $ les $o$-modules engendr\'es par $B^{+}V_{0}$  et $<(B-B^{+})V_{0}> $.

Lorsque $V$ se prolonge en une repr\'esentation admissible de $G$, les propri\'et\'es a), \ldots ,  e) sont vraies modulo la cat\'egorie
$ {\cal M}^{f}_{o}$  des $o$-modules finis.

\bigskip Lorsque $V\in  {\cal M}_{o-tor}(B)$, l'image de  $D_{V}(V_{0})$
par la dualit\'e de Pontryagin $M\mapsto M^{*}:=\Hom_{o}(M, L/o)$ est  un module sur  la $o$-alg\`ebre d'Iwasawa  $\Lambda_{o}(U_{0})$ de $U_{0}$ muni d'une action semi-lin\'eaire de $T_{0}t^{\bf N}Z$.  La dualit\'e de Pontryagin induit une \'equivalence entre  la cat\'egorie  $ {\cal M}^{tf}(\Lambda_{o}(U_{0}))$  des $\Lambda_{o}(U_{0})$-modules de type fini et la cat\'egorie des $M\in {\cal M}_{o-tor}(U_{0})$ tels que $M^{U_{0}, p_{L}=0}$ est fini  (lemme de Nakayama topologique).

\bigskip La m\'ethode pour passer du th\'eor\`eme technique \`a la partie 2) du th\'eor\`eme \ref{0} n'est pas nouvelle. Soit $O_{\cal E}$ l'anneau de Fontaine \'egal \`a la compl\'etion $p$-adique du localis\'e de  $\Lambda_{o}(U({\bf Z_{p}}))$  aux \'el\'ements non divisibles par $p$. Lorsque $F$ est de caract\'eristique $0$,  la  trace  $tr: O_{F}\to \bf Z_{p}$  
permet par produit tensoriel avec $O_{\cal E} \otimes_{\Lambda_{0}}- $ de passer des   $\Lambda_{0}$-modules aux $O_{\cal E}$-modules. Le produit tensoriel tue  $ {\cal M}^{f}_{o}(B_{0}Z)$, aussi  par le  th\'eor\`eme technique \ref{tech}, on obtient  un  foncteur contravariant de  $ {\cal M}_{o-tor}^{apf}(G) $   vers  la cat\'egorie des $(\varphi, \Gamma)$-modules \'etales de type fini  sur  $O_{\cal E}$.  Cette derni\`ere est \'equivalente \`a  ${\cal M}_{o}^{f}(\Gal _{\bf Q_{p}})$ par Fontaine.

\bigskip Nous n'obtenons  pas  des repr\'esentations  de $\Gal_{F}$ mais  de  $\Gal_{\bf Q_{p}}$.  La  repr\'esentation galoisienne obtenue  de $\Gal_{\bf Q_{p}}$ devrait  \^etre li\'ee \`a une repr\'esentation de $\Gal_{F}$. peut-\^etre par une induction tensorielle comme le sugg\`ere Breuil \cite{Breuil}.

\bigskip  Les  parties a) \`a d) du th\'eor\`eme technique, \'etendant \`a $F$ des constructions de P. Colmez, datent de  2006 et avaient \'et\'e expos\'ees dans un  cours de M2 de l'universit\'e de Paris 7,  et un colloque \`a Luminy.  La partie e),  li\'ee \`a un probl\`eme ouvert dans mes travaux avec P. Schneider pour des groupes de rang sup\'erieur d\'efinis sur $\bf Q_{p}$, est obtenue gr\^ace aux travaux et m\'ethodes d\'evelopp\'ees par  Y.Hu qui a \'etudi\'e de fa\c{c}on approfondie le cas o\`u $V$ et $V_{0}$ sont des  repr\'esentations irr\'eductibles de $G$ et de $KZ$ avec un caract\`ere central.  

\bigskip Je  remercie vivement  Pierre Colmez, Peter Schneider et Yongquan Hu de leurs  explications et des discussions 
que nous avons eues ensemble. Je remercie aussi  les organisateurs du colloque de Berlin en m\'emoire de Martin Grabitz en septembre 2009, de m'avoir permis d'y exposer ces r\'esutats. 

\section{Repr\'esentations lisses de pr\'esentation finie}

 Soit $o$ un anneau commutatif, $B$ un groupe localement profini, $B_{0} $ un sous-groupe ouvert de $B$. On s'int\'eresse \`a la cat\'egorie ${\cal M}_{o}(B)$ des repr\'esentations lisses du groupe $B$ sur des $o$-modules. La lissit\'e signifie que chaque vecteur est fixe par un sous-groupe ouvert de $B$. On  note 
 $$\ind^{B}_{B_{0}}\ : {\cal M}_o(B_{0}) \to {\cal M}_o(B)$$ le foncteur adjoint \`a gauche de la restriction de $B$ \`a $B_{0}$, appel\'e aussi l'induction compacte de $B_{0}$ \`a $B$.  C'est un foncteur exact.  Soit    $V_{0} \in  {\cal M}_o(B_{0})$.  On note $[1,V_{0}]$ le sous-espace de $\ind_{B_{0}}^{B}(V_{0})$ form\'e par les fonctions de support $B_{0}$ et on note $[1,v_{0}]$ celle de valeur $v_{0}\in V_{0}$ en $1$. 
  Soit $V\in {\cal M}_{o}(B)$ un quotient de $\ind^{B}_{B_{0}}(V_{0})$. On dit que 
 $$\ind^{B}_{B_{0}}(V_{0})\to V$$
 est une pr\'esentation de $V$. On dit que la pr\'esentation  est finie si $V_{0} \in  {\cal M}_o(B_{0})$  et le noyau $R_{V}(V_{0})\in {\cal M}_{o}(B)$ de la pr\'esentation sont des repr\'esentations de type fini.
 
 \begin{lemma} Une repr\'esentation $V\in {\cal M}_{o}(B)$ est de type fini si et seulement s'il existe $V_{0} \in  {\cal M}_o(B_{0})$  de type fini avec $V_{0}\subset V|_{B_{0}}$ engendrant $V$.
 \end{lemma}
 \begin{proof} Si $V$ est de type fini, engendr\'e par une partie finie $X\subset V$, alors la sous-repr\'esentation $V_{0}$ de $V|_{B_{0}}$ engendr\'ee par $X$ engendre $V$, et l'application naturelle $B$-\'equivariante $\ind_{B_{0}}^{B}(V_{0})\to V$ est une pr\'esentation de $V$. L'inverse est clair.
 \end{proof}
 
Un morphisme de pr\'esentations  d'une pr\'esentation $\ind^{B}_{B_{0}}(V_{0})\to V$  vers une pr\'esentation
$\ind^{B}_{B_{0}}(V'_{0})\to V$, 
est une application $B$-\'equivariante
$$f:\ind^{B}_{B_{0}}(V_{0})\to \ind^{B}_{B_{0}}(V'_{0})$$
 tel que le triangle form\'e avec $V$ soit commutatif. On note que 
  $$ f(R_{V}(V_{0})) \ \subset \ R_{V}(V'_{0}) $$
avec \'egalit\'e si $f$ est surjective.
 
 \begin{lemma}  \label{quot} Soit $f:\ind^{B}_{B_{0}}(V_{0})\to \ind^{B}_{B_{0}}(V'_{0})$ un morphisme surjectif entre deux pr\'esentations de $V$. Si le noyau $R_{V}(V_{0})\in {\cal M}_{o}(B)$ de la pr\'esentation $\ind^{B}_{B_{0}}(V_{0})\to V$ est de type fini, il en est de m\^eme du noyau $R_{V}(V' _{0})\in {\cal M}_{o}(B)$ de la seconde. La       
 r\'eciproque est vraie si le noyau de $f$ est de type fini.
 \end{lemma}
 
 \begin{proof} Evident.

\end{proof}

 Une partie $X$ de $V$ est appel\'ee g\'en\'eratrice si elle  engendre  le $o [B]$-module $V$. 
  Pour $W$ une partie de $V$  et $X$ une partie de $B$, on note $< XW>$ le sous-espace    de $V$ engendr\'e par les $xw$ pour tout $x\in X, w\in W$. 
Toute partie g\'en\'eratrice $X$ de $V$ d\'efinit naturellement une pr\'esentation
$$\ind^{B}_{B_{0}}(<B_{0}X>)\to V$$ de $V$.

  \begin{proposition} \label{fp} Si la pr\'esentation $\ind^{B}_{B_{0}}(V_{0})\to V$ de $V$ est finie,  si $W_{0}$ est
  un $o[B_{0}]$-sous-module de type fini de $\ind^{B}_{B_{0}}(V_{0})$ contenant $[1,V_{0}]$,  et si 
$W'_{0} $ est  l'image de $W_{0}$ dans $V$, 
alors la pr\'esentation $\ind^{B}_{B_{0}}(W'_{0})\to V$ est aussi finie.
 \end{proposition}
\begin{proof} 
a) Le cas $W_{0}=[1,V_{0}]$.
 
 Soit $f_{0}$ la restriction \`a $[1,V_{0}]$ de la pr\'esentation  $\ind^{B}_{B_{0}}(V_{0})\to V$. Il est clair que $V'_{0}=f_{0}(V_{0}) \in  {\cal M}_o(B_{0})$ est de type fini  et g\'en\'eratrice de $V$, et que 
 $$\ind_{B_{0}}^{B}(f_{0}):\ind^{B}_{B_{0}}(V_{0}) \to \ind^{B}_{B_{0}}(V'_{0}) $$ est un morphisme surjectif de pr\'esentations de $V$. Par le lemme \ref{quot},  le noyau de la seconde pr\'esentation est de type fini.
 
 b) Le cas g\'en\'eral.
 
Soit $f_{0}$ la restriction \`a $W_{0}$ de la pr\'esentation  $\ind^{B}_{B_{0}}(V_{0})\to V$, 
et 
$$\ind_{B_{0}}^{B}(W_{0})\to V$$ la pr\'esentation de $V$  de restriction $f_{0}$ \`a $[1,W_{0}]\simeq W_{0}$. 
L'application $B$-\'equivariante $$\phi : \ind_{B_{0}}^{B}(W_{0})\to \ind^{B}_{B_{0}}(V_{0})$$
d\'efinie par $\phi([1,w_{0}])=w_{0}$ pour tout $w_{0}\in W_{0}$ est un morphisme surjectif d\'eploy\'e de pr\'esentations de $V$,  car  l'inclusion $j_{0}: [1,V_{0}] \subset W_{0}$ induit un morphisme 
 $$\ind_{B_{0}}^{B}(j_{0}):\ind^{B}_{B_{0}}(V_{0}) \to \ind^{B}_{B_{0}}(W_{0}) $$ 
 de pr\'esentations de $V$ v\'erifiant
 $$\phi \circ  \ind_{B_{0}}^{B}(j_{0}) \ = \ \id \ .$$
(c'est \'evident car il suffit de le v\'erifier sur $[1,V_{0}] $). Donc 
le noyau de $\phi$ est isomorphe au  $o[B]$-module de type fini $\ind^{B}_{B_{0}}(W_{0}/[1,V_{0}])$. On d\'eduit du lemme \ref{quot}
que  la pr\'esentation $\ind_{B_{0}}^{B}(W_{0})\to V$ est finie. Par 
le cas a), la pr\'esentation $\ind^{B}_{B_{0}}(W'_{0})\to V$ est aussi finie.
  \end{proof}

   \begin{corollary} \label{cor1} Soit $V \in {\cal M}_{o}(B)$ de pr\'esentation finie et soit  $X$  une partie  finie de $V$ . Alors il existe une pr\'esentation  finie $\ind^{B}_{B_{0}}(V_{0})\to V$ avec $V_{0}\subset V|_{B_{0}}$ contenant $X$.
 \end{corollary}
 \begin{proof} On choisit une pr\'esentation finie
 $\ind^{B}_{B_{0}}(V_{0})\to V$ de $V$ puis   un $o[B_{0}]$-sous-module de type fini de $\ind^{B}_{B_{0}}(V_{0})$ contenant $[1,V_{0}]$,  tel que 
 l'image $W'_{0} $  de $W_{0}$ dans $V$ contienne $X$.  La pr\'esentation $\ind^{B}_{B_{0}}(W'_{0})\to V$ convient.
   \end{proof}
   
 \begin{corollary} Soit $V \in {\cal M}_{o}(B)$ admettant deux pr\'esentations finies   $\ind^{B}_{B_{0}}(V'_{0})\to V$ et $\ind^{B}_{B_{0}}(V''_{0})\to V$. Alors il existe une pr\'esentation  finie $\ind^{B}_{B_{0}}(V_{0})\to V$ de $V$  avec $V_{0}\subset V|_{B_{0}}$, et des morphismes de pr\'esentation
 $$\ind^{B}_{B_{0}}(V'_{0}) \to \ind^{B}_{B_{0}}(V_{0}) \ \ ,  \ \ \ind^{B}_{B_{0}}(V''_{0}) \to \ind^{B}_{B_{0}}(V_{0})$$
 de $V$.
 \end{corollary}
 \begin{proof} Par le corollaire \ref{cor1}, il existe un $o[B_{0}]$-sous-module $V_{0}$ g\'en\'erateur de $V$ contenant les images de $V'_{0} $ et de $V''_{0}$ dans $V$, tel que la pr\'esentation  $\ind^{B}_{B_{0}}(V_{0})\to V$ de $V$  soit finie. Les applications $f'_{0}:V'_{0}\to V_{0}$ et $f''_{0}:V''_{0}\to V_{0}$ induisent des morphismes 
 de pr\'esentations  $$\ind^{B}_{B_{0}}(V'_{0}) \to \ind^{B}_{B_{0}}(V_{0}) \ \ ,  \ \ \ind^{B}_{B_{0}}(V''_{0}) \to \ind^{B}_{B_{0}}(V_{0})$$
 de $V$.
 
   \end{proof}

\begin{proposition}\label{suite}  On suppose que la cat\'egorie  $ {\cal M}_o(B_{0})$ est noetherienne. Soit 
  $$ 0 \to V' \to V \to V'' \to 0$$ une suite exacte dans   ${\cal M}_o(B)$.
  
  a) Si $V',V''$ sont de pr\'esentation finie,  alors il existe des $o[B_{0}]$-sous-modules $V'_{0}, V_{0}, V''_{0}$ g\'en\'erateurs de $V', V, V'',$ formant une suite exacte
  $$0 \to V'_{0} \to V_{0} \to V''_{0} \to 0 \ ,$$
 tels que les pr\'esentations associ\'ees de $V',V,V''$ soient finies.

 b) Si $V'$  est de type fini et si $V$ est de pr\'esentation finie, alors $V''$ est de pr\'esentation finie.

\end{proposition}
\begin{proof} Pr\'eliminaires.  On identifie $V'$ \`a une sous-repr\'esentation de $V$.
Soit  $W''_{0}$  une sous-repr\'esentation de $V|_{B_{0}}$  d'image  $V''_{0}$ dans $V''|_{B_{0}}$  g\'en\'eratrice de $V''$, et soit 
$V'_{0} \in {\cal M}_{o}(B_{0})$  contenue dans $V'|_{B_{0}}$ g\'en\'eratrice de $V'$ et  contenant $W''_{0}\cap V'$. Alors  la sous-repr\'esentation $V_{0}:=V'_{0} +W''_{0}$  de $V|_{B_{0}}$ engendre $V$, et l'on a  la suite exacte dans  ${\cal M}_{o}(B_{0})$,
  $$0 \to V'_{0} \to V_{0} \to V''_{0} \to 0 \ . $$
  Son image par $\ind_{B_{0}}^{B}$ est la suite exacte dans $  {\cal M}_{o}(B)$, 
$$0 \to \ind_{B_{0}}^{B}(V'_{0})\to \ind_{B_{0}}^{B}(V_{0} ) \to \ind_{B_{0}}^{B}(V''_{0} ) \to 0 \ $$
d'image par les morphismes de pr\'esentation  correspondants,  la  suite exacte
$$ 0 \to V' \to V \to V'' \to 0 \ . $$ 
Les noyaux des morphismes de pr\'esentation  forment une suite exacte dans $  {\cal M}_{o}(B)$, 
$$0 \to R_{V'}(V'_{0})\to R_{V}(V_{0} ) \to R_{V''}(V''_{0} ) \to 0 \ .$$

Montrons la propri\'et\'e a). Si  $V', V''\in {\cal M}_o(B)$ sont de type fini, on peut choisir $V''_{0}, W''_{0}  \in {\cal M}_{o}(B_{0})$ de type fini; comme la cat\'egorie ${\cal M}_{o}(B_{0})$  est noetherienne, $W''_{0}\cap V'$ est de type fini. On peut choisir $V'_{0}\in  {\cal M}_{o}(B_{0})$ de type fini, donc  $V_{0}\in  {\cal M}_{o}(B_{0})$ de type fini. Si  $V',V''\in {\cal M}_o(B)$ sont  de pr\'esentation finie, on peut  de plus choisir 
$V'_{0},V''_{0}$ tels que
 $R_{V'}(V'_{0}),R_{V''}(V''_{0} )\in {\cal M}_o(B)$ sont de type fini par le corollaire \ref{cor1};  il en est de m\^eme de $R_{V}(V_{0} )$.

Montrons la propri\'et\'e b). On choisit par la proposition \ref{fp} une pr\'esentation finie 
$\ind^{B}_{B_{0}}(V_{0})\to V$ de $V$ telle que $V_{0}\subset V|_{B_{0}}$ contient l'image 
d'un syst\`eme fini de g\'en\'erateurs  de $V'$. On d\'efinit  $V''_{0}$ comme l'image de $V_{0}$ dans $V''$. Le noyau  $V'_{0} \in {\cal M}_{o}(B_{0})$ de la surjection  $V_{0}\to V''_{0}$ est  contenu dans $V'|_{B_{0}}$ et engendre  $V'$. On est dans la situation des pr\'eliminaires. Comme  $R_{V}(V_{0} )$ est de type fini, il en est de m\^eme de $R_{V''}(V''_{0} )$.

\end{proof}
La cat\'egorie $ {\cal M}_o(B_{0})$  est noetherienne si et seulement si 
la cat\'egorie $ {\cal M}_o^{tf}(B_{0})$ 
des repr\'esentations de type fini  dans  ${\cal M}_o(B_{0})$ est une sous-cat\'egorie de Serre : ferm\'ee par 
sous-objets, quotients et extensions, i.e. pour toute suite exacte dans  ${\cal M}_o(B_{0})$,
$$0\to V'_{0}\to V_{0}\to V''_{0}\to 0$$
$V'_{0} $ et $V''_{0}$ sont de type fini si et seulement si $V_{0}$ est de type fini.

\begin{definition} Lorsque la cat\'egorie ${\cal M}_o(B_{0})$ est noetherienne, deux repr\'esentations de ${\cal M}_o(B_{0})$ isomorphes dans   la cat\'egorie quotient 
 $${\cal M}_o(B_{0})/ {\cal M}_o^{tf}(B_{0})\  .$$
  seront dites isomorphes  modulo $ {\cal M}_o^{tf}(B_{0})$. 
Un suite $V_{1} \to V_{2} \to V_{3}$ dans  ${\cal M}_o(B_{0})$ 
telle que l'image du premier morphisme est isomorphe au noyau du second morphisme modulo $ {\cal M}_o^{tf}(B_{0})$, sera dite exacte modulo $ {\cal M}_o^{tf}(B_{0})$. 
\end{definition}

\begin{definition} On dit que $V\in  {\cal M}_{o}(B)$ est admissible si $V^{H}$ est un $o$-module de type fini pour tout sous-groupe ouvert compact $H$ de $B$. 
\end{definition}

 \begin{lemma} \label{ext} Soit  $o'$ un anneau commutatif contenant $o$ et libre comme $o$-module, et soient $V\in {\cal M}_{o}(B)$  et $V_{o'}:=o'\otimes _{o}V \in {\cal M}_{o'}(B)$.
 Alors $V$ est  admissible (resp. de pr\'esentation finie), si et seulement si $V_{o'}$ est admissible (resp. de pr\'esentation finie).
 \end{lemma}
\begin{proof} Soit $(e_{i})_{i\in I}$  une base  de $o'/o$.
On a $$V_{o'}=\oplus _{i\in I}\ (e_{i} \otimes V)\quad .$$ 
Pour un sous-groupe ouvert $H$ de $B$ on a  $V_{o'}^{H}= \oplus_{i\in I} \ (e_{i} \otimes V^{H})$.  
Donc,   $V'$ admissible  est \'equivalent \`a $V$ admissible.

Soit $V_{0} \in    {\cal M}_{o}(B_{0})$ contenu dans $V|_{B_{0}}$ g\'en\'eratrice de $V$, et $R_{V}(V_{0})$ le noyau de la pr\'esentation associ\'ee $ \ind_{B_{0}}^{B}V_{0}\to V$. 

Alors $V'_{0}:=o'\otimes _{o}V_{0}$ est une sous-repr\'esentation de $V'|_{B_{0}}$  g\'en\'eratrice de $V'$. Soit  $R_{V'}(V_{0}')$ le noyau de la pr\'esentation associ\'ee de $V'$. On a 
   $$V'_{0}= \oplus_{i\in I} \ e_{i} \otimes V_{0} \quad, \quad R_{V'}(V'_{0})=o'\otimes_{o}R_{V}(V_{0}) \ . $$
Donc   la pr\'esentation de $V'$ associ\'ee \`a $V'_{0}$ est finie, si et seulement si la pr\'esentation de $V$ associ\'ee \`a $V_{0}$ est finie. 
\end{proof}

\begin{definition} Soit $Z$ un sous-groupe ferm\'e du centre de $B$. On dit que $V\in  {\cal M}_{o}(B)$  est $Z$-localement finie si la sous-repr\'esentation de $Z$ engendr\'ee par $v$ dans $V|_{Z}$ est de type fini pour tout \'el\'ement $v\in V$. 
\end{definition}

C'est une propri\'et\'e un peu plus g\'en\'erale que d'avoir un caract\`ere central. Une repr\'esentation admissible $V\in {\cal M}_{o}(B)$ est $Z$-localement finie. Lorsque $B_{0}$ est un sous-groupe ouvert compact de $B$, une repr\'esentation $Z$-localement finie $V_{0}\in  {\cal M}_{o}^{tf}(B_{0}Z)$   est un $o$-module de type fini. 

\bigskip On note $ {\cal M}_{o-tor}$  la cat\'egorie   des $o$-modules de torsion. 
  \begin{lemma} \label{le3} On suppose que $o$ est un anneau de valuation discrete, d'uniformisante $p_{o}$.  
  
  1) Soit $V\in {\cal M}_{o}(B)$ tel que $p_{o} ^{n}V=0$ pour   $n\in \bf N$ (en particulier une repr\'esentation $V\in {\cal M}_{o-tor}^{tf}(B)$). Alors, 
 $V^{p_{o}=0} $ est admissible 
  si et seulement si $V$ est admissible. 
  
  2) Soit $V_{0}\in  {\cal M}_{o-tor}^{tf}(B_{0}Z)$ repr\'esentation $Z$-localement finie. Lorsque $B_{0}$ est un sous-groupe ouvert compact de $B$, alors $V_{0}$ est fini et  il existe un anneau de valuation discr\`ete $o'$, libre de rang fini sur $o$, tel que les sous-quotients irr\'eductibles de $V_{0,o'}$ sont absolument irr\'eductibles (en particulier $Z$ agit par un caract\`ere).

  \end{lemma}
  
  \begin{proof} 1)  Si $V$ est admissible,  alors $V^{p_{o}=0}$ est  admissible car l'anneau $o$ est noetherien. 
  
  Si $V^{p_{o}=0} $ est admissible, alors la sous-repr\'esentation 
  $V':=V^{p_{o}^{2}=0} $ est admissible car $p_{o}V'$ contenue dans 
   $V^{p_{o}=0} $ est admissible, la multiplication par $p_{o}$ induit une suite exacte
   $$0 \to V^{p_{o}=0} \to V'\to p_{o}V' \to 0 \ , $$
   et l'admissibilit\'e est stable par extensions. En continuant cet argument on montre que $V^{p_{o}^{i}=0} $ est admissible pour tout entier $i\geq 1$, donc $V$ est admissible si $p_{o} ^{n}V=0$.
   
   2) Evident. Un quotient fini de $B_{0} $ agit sur $V_{0}$.
     \end{proof}

\begin{remark} \label{re3} Soit $V\in {\cal M}_{o-tor}^{tf}(B)$ une repr\'esentation  $Z$-localement finie.
La partie 2) du lemme \ref{le3} permet souvent de ramener  une question sur $V$, insensible aux extensions et \`a l'extension des scalaires, au cas o\`u $V$ est un quotient de $\ind_{B_{0}Z}^{B}(V_{0})$ avec $V_{0}\in  {\cal M}_{o-tor}^{tf}(B_{0}Z)$ absolument irr\'eductible o\`u $Z$ agit par un caract\`ere.
\end{remark}

En effet, on choisit  $V_{0}\in  {\cal M}_{o-tor}^{tf}(B_{0}Z)$ contenue dans $V|_{B_{0}Z} $ et g\'en\'eratrice. Puis $o'$ libre de rang fini sur $o$ tel que  $V_{0,o'}$  admette une suite de Jordan-H\"older   de quotients abolument irr\'eductibles $W_{o}^{(i)}$ pour $1\leq i \leq n$. Prenons l'image d'une telle suite par
 $\ind_{B_{0}Z}^{B}$ qui est un foncteur exact, puis prenons l'image dans $V_{o'}$ de la filtration obtenue de  $\ind_{B_{0}Z}(V_{0,o'})$ par l'application surjective canonique. On obtient ainsi une filtration finie de $V_{o'}$ dont les quotients sont des quotients de $\ind_{B_{0}Z}^{B}(W_{o}^{(i)})$ pour $1\leq i \leq n$.

\section  {Repr\'esentations  des monoides  $B^{+}$ et $(B^{+} )^{-1}$}

\subsection {Les d\'efinitions} Soit  $G:=GL(2,F)$ o\`u $F$ est un corps local non archim\'edien de caract\'eristique r\'esiduelle $p$. L'anneau des entiers $O_{F}$ de $F$ est local, d'id\'eal maximal engendr\'e par $p_{F}$, et de  corps r\'esiduel   $\bf F_{q}$ fini de cardinal $q$. 

\bigskip On consid\`ere les sous-groupes ferm\'es suivants de $G$ :  le centre $Z$ form\'e par  les matrices $x\id_{2}$ pour $x\in F^{*} $,  le sous-groupe triangulaire sup\'erieur $B=B(F)$,    le sous-groupe diagonal $T=T(F)$,  
le radical unipotent $U=U(F)$ de $B$,  les groupes compacts  $U_{0}:=U(O_{F})\ , \ T_{0}:=T(O_{F})$  et $B_{0}:= U_{0}T_{0}=B(O_{F})$.  On pose
$$ t :=\pmatrix {p_{F}&0 \cr 0&1} \  . $$ 
On consid\`ere le monoide $T^{+}:=T_{0}Zt^{\bf N}$, le monoide $B^{+}:=U_{0}T^{+} = B_{0}Zt^{\bf N}= B_{0}Zt^{\bf N}B_{0}$, et  $(B^{+})^{-1}$ le monoide inverse de $B^{+}$ form\'e des $b^{-1}$ pour $b\in B^{+}$.

     \begin{lemma}\label{5} Si $b^{-1}\in B^{+}$ alors $b (B-B^{+})  \subset (B-B^{+}) $.
  \end{lemma}
  
  \begin{proof} Evident.
  \end{proof}

Soit $V\in {\cal M}_{o} (B)$  et $V_{0}\in {\cal M}_{o} (B_{0 }Z)$ contenu dans $V|_{B_{0}Z}$ et g\'en\'erateur de $V$. On note $R_{V}(V_{0})$ le noyau de la pr\'esentation $\ind_{B_{0}Z}^{B}(V_{0})\to V$.

    \begin{definition} On note
   $$ \Delta_{V}(V_{0}):=  <B^{+}V_{0} >  \cap <(B-B^{+})V_{0}> \ \ .$$
   \end{definition} 

  Le premier terme du membre de droite est stable par $B^{+}$ et le second terme est stable par $(B^{+})^{-1}$ par le lemme \ref{5}. Leur intersection $ \Delta_{V}(V_{0}) $ est stable par $B_{0}Z=B^{+}\cap (B^{+})^{-1}$. Elle est nulle si  $V=\ind^{B}_{B_{0}Z}(V_{0})$ (on identifie $V_{0}$ et $[1,V_{0}]$).
   On a
   $$V= <B^{+}V_{0}>  + <(B-B^{+})V_{0}>  \  .$$

 \begin{definition} On note  $$D_{V}(V_{0}):=  {<B^{+}V_{0}>   \over \Delta_{V} (V_{0}) }$$ muni de son action de $(B^{+})^{-1}$ provenant des isomorphismes.
  $$D_{V}(V_{0})\ \simeq \ {V\over   <(B-B^{+})V_{0}> } \  \simeq { \ind_{B_{0}Z}^{B}(V_{0})\over R_{V}(V_{0})+ (B-B^{+})[1,V_{0}]}\ . $$
   \end{definition}

 \bigskip Le caract\`ere de Teichm\"uller permet d'identifier $U(q):=U(\bf F_{q})$ \`a un syst\`eme de repr\'esentants dans $U_{0}$ de $U_{0} /tU_{0}t^{-1}=U_{0}/U(p_{F}O_{F})$.

\begin{definition}\label{et} On dit qu'une repr\'esentation $D\in  {\cal M}_{o}((B^{+})^{-1})$ est \'etale si l'application 
 \begin{equation} \label{etale}
 d \mapsto ((ut)^{-1}d)_{u\in U(q)}\ \ : \ \ D \ \to \ \oplus_{u\in U(q)} D
 \end{equation}
 est bijective.
 \end{definition}

Cette d\'efinition est duale de la d\'efinition usuelle d'une repr\'esentation \'etale  $D\in  {\cal M}_{o}(B^{+})$ : l'application  
$$ (d_{u})_{u\in U(q)} \mapsto \sum_{u}ut d_{u}\ : \ \oplus_{u\in U(q)} D\ \to D$$
est bijective.

\begin{lemma} La cat\'egorie  ${\cal M}_{o}^{et}((B^{+})^{-1})$  des repr\'esentations \'etales $D\in  {\cal M}_{o}((B^{+})^{-1})$ est ab\'elienne.
\end{lemma}
\begin{proof}
Soit $f:D_{1}\to D_{2}$ un morphisme dans ${\cal M}_{o}^{et}((B^{+})^{-1})$  des repr\'esentations \'etales.
Il est clair que le noyau et l'image de $f$ sont \'etales.
\end{proof}

 \begin{lemma} \label{3}   Lorsque l'anneau commutatif $o$ est noetherien, 
 la cat\'egorie  ${\cal M}_o(B_{0}Z)$ est noetherienne.
\end{lemma}
\begin{proof} 
 On a $B_{0}Z=B_{0}(p_{F}\id_{2})^{\bf Z}$ et la compacit\'e de $B_{0}$ implique qu'une repr\'esentation $V\in {\cal M}_o(B_{0}Z)$ est de type fini si et seulement si sa restriction au sous-groupe $(p_{F}\id_{2})^{\bf Z}$ est de type fini. La cat\'egorie ${\cal M}_{o}((p_{F}\id_{2})^{\bf Z})$ s'identifie \`a la cat\'egorie des modules de 
 l'anneau de polyn\^omes $o[{\bf Z}]$  qui est noetherien car $o$ est noetherien. Donc la cat\'egorie ${\cal M}_o^(B_{0}Z)$ est noetherienne. 
 \end{proof}

  \subsection{Les propri\'et\'es  de $D_{V}(V_{0})$ modulo un $o[B_{0}Z]$-module de type fini}

Dor\'enavant, $o$ est un anneau noetherien. Soit 
  ${\cal M}_o^{tf}(B_{0}Z)$ la  sous-cat\'egorie  de Serre des repr\'esentations de type fini  dans  ${\cal M}_o(B_{0}Z)$.  Nous  allons montrer l'existence,  \`a une ``petite erreur'' pr\`es, i.e. modulo   ${\cal M}_o^{tf}(B_{0}Z)$,  d'un foncteur exact contravariant   $$V\mapsto  D_{V}(V_{0}) \equiv <B^{+}V_{0}>$$  des repr\'esentations  $V\in {\cal M}_o(B)$ de pr\'esentation finie vers  ${\cal M}_{o}^{et} ((B^{+})^{-1})$. 

\begin{theorem}  On suppose $o$ noetherien. Soit $V\in {\cal M}_o(B)$ de pr\'esentation finie  $\ind_{B_{0}Z}^{B}(V_{0})\to V$ avec 
$V_{0}\subset V|_{B_{0}Z}$.  Alors    modulo $ {\cal M}_o^{tf}(B_{0}Z)$,  

a) $\Delta_{V} (V_{0})\equiv 0$ , 

b) La repr\'esentation $D_{V}(V_{0})\in {\cal M}_o((B^{+})^{-1})$ est \'etale et  ne d\'epend pas du choix de $V_{0}$.

c) 
 Une suite exacte de repr\'esentations de pr\'esentation finie dans   ${\cal M}_o(B)$,
 $$ 0 \to V' \to V \to V'' \to 0$$
 et $(V'_{0}, V_{0}, V''_{0})$ comme dans la proposition \ref{suite}, 
induisent une suite exacte dans ${\cal M}_o((B^{+})^{-1})$ modulo $ {\cal M}_o^{tf}(B_{0}Z)$,
  $$ 0 \to D_{V'}(V'_{0}) \to  D_{V}(V_{0}) \to  D_{V''}(V''_{0}) \to 0 \ .$$
 
 \end{theorem}

Le reste de ce chapitre est consacr\'e \`a la d\'emonstration. 

\bigskip   Nous d\'ecrivons la repr\'esentation $\Delta_{V}(V_{0})$ de $B_{0}Z$, comme un quotient de sous-espaces de $\ind_{B_{0}Z}^{B}(V_{0})$. Pour cela,
on consid\`ere la projection  $S^{-}_{V}(V_{0})$ de $R_{V}(V_{0})$ dans $<(B-B^{+}) [1,V_{0}] >$
\begin{equation} \label{proj}
R_{V}(V_{0})+< B^{+} [1,V_{0}] > = S^{-}_{V}(V_{0}) \oplus < B^{+}  [1,V_{0}] > \ ,
\end{equation}
  l'intersection 
$R^{+}_{V}(V_{0}):= \ R_{V}(V_{0})\  \cap \ <B^{+}[1,V_{0}] >  $,  ainsi que la projection  $S_{V}^{+}(V_{0})$ de $R_{V}(V_{0})$ dans $< B^{+} [1,V_{0}] >$ et   l'intersection 
$R^{-}_{V}(V_{0}):= \ R_{V}(V_{0})\  \cap \ <(B-B^{+})[1,V_{0}] >  \    .$

\begin{lemma} \label{0'} On a les suites exactes naturelles
$$0 \to R^{\e}_{V}(V_{0}) \to S^{\e}_{V}(V_{0}) \to \Delta_{V}(V_{0}) \to 0 \ $$
avec $\e = +$ ou $\e = -$.
\end{lemma}

\begin{proof}  Il est clair que l'image de  $S^{\e}_{V}(V_{0})$  par l'application canonique est contenue dans $\Delta_{V}(V_{0}) $. Inversement, soit  $x \in  \Delta_{V}(V_{0}) $. On choisit  $y\in B^{+}[1,V_{0}]$ et $z\in (B-B^{+})[1,V_{0}]$ d'image $x$. On a  $y-z\in R_{V}(V_{0})$. On en d\'eduit que $y\in S^{+}_{V}(V_{0})$ et $z\in S^{-}_{V}(V_{0})$. 
\end{proof}

 \begin{proposition} \label{1} Soient $V\in {\cal M}_o(B)$ et   $V_{0} , V'_{0} \in {\cal M}_o(B_{0}Z)$  contenues dans $V|_{B_{0}Z}$ et g\'en\'eratrices de  $V$. 
 
  a) Si $V_{0}$ et $V'_{0} $ sont de type fini, alors  $<B^{+}V_{0}> $  et  $<B^{+}V'_{0}>$ sont isomorphes  modulo $ {\cal M}_o^{tf}(B_{0}Z) $.
  
 b) Si $R_{V}(V_{0})$ est de type fini, alors  $ \Delta_{V} (V_{0}) \in  {\cal M}_o^{tf}(B_{0}Z) $.
   \end{proposition}

\begin{proof} 

a) Nous allons montrer qu'il existe  $M\in  {\cal M}_o^{tf}(B_{0}Z)$ contenu dans $V|_{B_{0}Z}$ tel que   
$$<B^{+}V_{0}> +M \  = \  <B^{+} V'_{0}> +M \ . $$

 a1)  Comme $V_{0}  \in {\cal M}_o^{tf}(B_{0}Z)$,  il existe une partie  finie $(Y,X'_{0})$ de $(B,V'_{0})$ telle que $<YX'_{0}>$ contienne une partie g\'en\'eratrice finie $X_{0}$ de $V_{0}$. Il existe
$n\in \bf N$ tel que $t^{n}Y\subset B^{+}$. On a donc 
$$<B^{+}t^{n}V_{0}> \ = \ <B^{+}t^{n}B_{0}Z X_{0}>\ = \ <B^{+}t^{n}X_{0}>\ \subset \ <B^{+} V'_{0}> \ .$$ L'ensemble $B^{+}-B^{+}t ^{n}$ est l'union finie disjointe des ensembles $B_{0}Zt^{i}$  pour $0\leq i \leq n-1$. On a 
 $$<B^{+}V_{0}> \ = \  <B^{+}t^{n}V_{0}> \ + \ M_{1} \quad , \quad M_{1} :=\sum_{i=0}^{n-1}  <B_{0}Zt^{i}V_{0}> \ . $$
 Clairement $M_{1}\in  {\cal M}_o^{tf}(B_{0}Z)$ est contenu dans $V|_{B_{0}Z}$ et   $$<B^{+}V_{0}> \ \subset \  <B^{+} V'_{0}> \ +\ M_{1} \ . $$ 

a2)  Comme $V'_{0}  \in {\cal M}_o^{tf}(B_{0}Z)$, par sym\'etrie en $V_{0}, V'_{0}$, on a aussi  $$<B^{+}V'_{0}> \ \subset \  <B^{+} V_{0}> +M_{2}$$ o\`u $M_{2}\in  {\cal M}_o^{tf}(B_{0}Z)$ est contenu dans $V|_{B_{0}Z}$. 

On prend $M:=M_{1}+M_{2}$.

\bigskip b)  
On choisit  une partie finie $(Y, X_{0})$ de $(B, V_{0})$ telle que $<Y[1,X_{0}] >$ contienne une partie g\'en\'eratrice finie $X $ de $R_{V}(V_{0})$.

  Soit $B_{1}$ l'ensemble des $b\in B$ tels que $bY \subset B^{+}$. Pour tout $b\in B_{1}$ nous avons $bX\subset R^{+}(V_{0})$. 
     
   Soit $B_{2}$ l'ensemble des $b\in B$ tels que $bY \subset (B-B^{+})$. Pour tout $b\in B_{2}$ la projection de  $bX $ dans $B^{+}[1,V_{0}]$ est nulle. 

    Pour tout $y,z \in Y$ soit $B_{y,z}$  l'ensemble des $b\in B$ tels que 
    $$by  \subset B^{+} \ \ , \ \ bz \subset (B-B^{+}) \ \ .$$
    Le compl\'ementaire $B_{3}$ de $B_{1}\cup B_{2}$ dans $B$ est l'union finie des $B_{y,z}$ pour tout $y,z\in Y$. 
    
  Pour $b\in B$,  notons $u_{b} \in U $ et $t_{b} $ les composantes de $b $ dans la d\'ecomposition $B=UT$.  Soit  $b\in B$ appartenant \`a $B_{y,z}$, i.e.  $u_{b},t_{b}$ satisfont les 3 conditions:
  
  1) \ $t_{b}t_{y}\in T^{+} $ ,
  
  2) \ $u_{b}t_{b}u_{y}t_{b}^{-1} \in  U_{0}$ , 
  
  3) \ $t_{b}t_{z} \in T^{+} $ implique $u_{b}t_{b}u_{z}t_{b}^{-1}\  \not\in \  U_{0}$ \ . 

\noindent Pour $y\in Y$ notons $n_{y}\in \bf Z$  l'entier tel  que $ T^{+}t_{y}^{-1}= T^{+}t^{n_{y}}$.
 Choisissons un entier $n_{Y} \in \bf Z$  
 tel que 
 
 a)   $t^{n_{Y}}t_{y }\in T^{+}$ pour tout $y\in Y$,
 
 b)  $t^{n_{Y}}u_{y}t^{-n_{Y}} \in U_{0}$   pour tout $y\in Y$.
   
 \noindent   
  
  On ne peut pas avoir $t_{b}\in   T^{+}t ^{n_{Y}}$, car pour  $t_{b}\in   T^{+}t ^{n_{Y}}$ les conditions 2) et 3) sont 
$$u_{b}\in U_{0} \ \ , \ \ t_{b}t_{z}\not\in T^{+} \ \ ,$$
ce qui est incompatible. Par la condition 1), on a 
$$t_{b}\ \in \ T^{+}t^{n_{y}}-  T^{+}t ^{n_{Y}}\ \ =\ \  \cup_{i=n_{y}}^{n_{Y}-1} T_{0}Z t^{i} \quad . $$ 
 Donc
  $$A_{y}:=\{ t_{b}u_{y}t_{b}^{-1} \quad | \quad t\in T^{+}t^{n'_{Y}}-  T^{+}t ^{n_{Y}}\}$$
  est  compact, et
  $$B_{y}:= \{ u_{b}\in U \quad | \quad u_{b}A_{y} \subset U_{0}\} $$ 
   est compact et contient $U_{0}$. Comme $U$ est ouvert dans $U$, il existe un ensemble fini $V$ dans $U$ tel que
  $$B_{y,z}\quad \subset \quad \cup_{i=n'_{Y}}^{n_{Y}-1} B_{0}V Z t^{i} \quad  \subset B_{0}ZC, $$ 
  o\`u $n'_{Y} = \inf_{y\in Y} n'_{y}$ et $C$ est  fini.  Comme l'action est lisse, $X':=C[1,X_{0} ] $ est fini. On en d\'eduit que 
  \begin{equation}\label{S+}
{S^{+}_{V}(V_{0})\quad  \subset \quad  R^{+}_{V}(V_{0}) } \ + \  < B_{0}Z X'>\ .  
\end{equation}
Par le lemme   \ref{0'},  on d\'eduit $\Delta_{V}(V_{0})\ \subset \  <B_{0}Z X'>$ est de type fini dans ${\cal M}_{o}(B_{0}Z)$.
   
\end{proof}

 \begin{proposition} \label{11}  Soit 
  $ f:V' \to V$ un morphisme dans   ${\cal M}_o(B)$,  soit $V'_{0} \in   {\cal M}_o(B_{0}Z)$ contenue dans $V'|_{B_{0}Z} $ de type fini et g\'en\'eratrice de $V'$, et soit  $\ind_{B_{0}Z}^{B}(V_{0})\to V$  une pr\'esentation  finie de $V$ avec $V_{0}\subset V|_{B_{0}Z}$. Alors $$f (<B^{+}V'_{0}>) \ \ , \ \ (f(V')\ \cap \  <B^{+}V_{0}>)$$ sont isomorphes  modulo $ {\cal M}_o^{tf}(B_{0}Z)$.
  
 \end{proposition}

\begin{proof} 
  a) Il existe $V_{0}$ comme dans l\'enonc\'e contentant $f(V'_{0})$ par le corollaire \ref{cor 1}. 
  Par la proposition \ref{1} a), on peut se ramener \`a ce cas.
  Supposons donc  $f(V'_{0})\subset V_{0}$. Alors
  $$f (<B^{+}V'_{0}>)  \quad  \subset \quad  (f(V')\  \cap \ <B^{+}V_{0}>)  \ . $$
Le terme de droite est \'egal \`a
$$ <B^{+} f(V'_{0})> \ +\ (<(B-B^{+}) f(V'_{0})>  \ \cap \ <B^{+}V_{0}>) \ . $$
Il est contenu dans $  <B^{+} f(V'_{0})> +\ \Delta _{V}(V_{0})$. On a donc
  $$f (<B^{+}V'_{0}>) \ + \ \Delta _{V}(V_{0})\ = \ (f(V') \cap <B^{+}V_{0}>) \ +  \ \Delta _{V}(V_{0}) \ .$$

Par la proposition \ref{1} b), $\Delta _{V}(V_{0})  \in   {\cal M}_o^{tf}(B_{0}Z)$.
\end{proof}

\begin{corollary}  Soient 
  $ 0 \to V' \to V \to V'' \to 0$ une suite exacte dans   ${\cal M}_o(B)$  
  telle $V'$ soit de type fini et $V$ de pr\'esentation finie. 
  Soit  $\ind_{B_{0}Z}^{B}(V_{0})\to V$ une pr\'esentation  finie de $V$ avec $V_{0} \subset V|_{B_{0}Z}$  telle que dans la suite exacte induite dans  ${\cal M}_o(B_{0}Z)$ 
    $$ 0 \to V' _{0}\to V_{0} \to V_{0}'' \to 0$$ 
  $V'_{0}$ engendre $V'$.
  Alors la suite
$$ 0 \ \to \  <B^{+}V'_{0}>  \ \to \  <B^{+}V_{0}> \ \to \  <B^{+}V''_{0}>\  \to \  0$$
est exacte, modulo ${\cal M}_o^{tf}(B_{0}Z)$.
\end{corollary}

\begin{proof}  Par la preuve de la proposition \ref{suite}, $V_{0}$ existe bien. Notons $f$ l'application de $V'$ dans $V$. On voit facilement que la derni\`ere suite est exacte si et seulement si
 $$(<B^{+}V_{0}> \  \cap \  f(V')) \ \subset \ f( <B^{+}V'_{0}>). $$
 On applique la proposition \ref{11}.

\end{proof}

  \begin{lemma} \label{111} Soit $V\in {\cal M}_o(B)$  de presentation finie $\ind_{B_{0}Z} ^{B}(V_{0}) \to V$ avec  $V_{0}\subset V|_{B_{0}Z}$. Alors  l'application:
$$d\mapsto ((ut)^{-1}d)_{u\in U(q)} \ : \ \ D_{V}(V_{0}) \ \to \ \oplus_{u\in U(q)} D_{V}(V_{0})$$
est surjective, et son noyau appartient \`a $ {\cal M}_o^{tf}(B_{0}Z)$.
  \end{lemma}
  
\begin{proof} Posons $C:= R_{V}(V_{0})+ (B-B^{+})[1,V_{0}]$ pour simplifier la notation. Alors  $D_{V}(V_{0})$ est le quotient de  $\ind_{B_{0}Z}(V_{0})$ par $C$ et le noyau de l'application du lemme est 
$${ \cap_{u\in U(q)} \ ut C\over C} \ .$$
Calculons $utC$. Par (\ref{proj}) 
$$
ut C \ = \ R_{V}(V_{0}) + (B-utB^{+})[1,V_{0}] \ = \ utS_{V}^{+}(V_{0}) \oplus (B-utB^{+})[1,V_{0}] \ . 
$$
 En d\'eveloppant le membre de droite  en utilisant  l'union disjointe $$B^{+}t =\cup_{u\in U(q)}utB^{+} \ , $$ on obtient 
 \begin{equation}\label{inter}ut C \ = \  (B-B^{+}t)[1,V_{0}]  \oplus  utS_{V}^{+}(V_{0})\oplus   _{u'\neq u , u'\in U(q)}u'tB^{+}[1,V_{0}] \ ,\end{equation}
 $$\cap_{u\in U(q)} \ ut C \ = \ (B-B^{+}t) [1,V_{0}] \oplus_{u\in U(q)} ut S_{V}^{+}(V_{0}) \ . $$
Comme $R_{V}(V_{0})\in {\cal M}_o(B)$ est de type fini, on la relation  (\ref{S+}). Il existe une partie finie $X'$ de $B[1,V_{0}]$ telle que
 $utS^{+}_{V}(V_{0})$  est contenu dans
$$ R^{+}_{V}(V_{0})  +   < B_{0}Z tX' >  \ . 
$$ 
Ceci ne d\'epend pas du choix de $u\in U(q)$ et 
$$ \cap_{u\in U(q)} \ ut C \quad  \subset \quad   (B-B^{+}t) [1,V_{0}] + R^{+}_{V}(V_{0})\   +\ < B_{0}Z tX' >  \ .$$
Comme $B^{+}- B^{+}t=B_{0}Z$, 
$$ \cap_{u\in U(q)} \ ut C \quad  \subset \quad C+ [1,V_{0} ] + < B_{0}Z tX' >  \quad \equiv \quad C \quad $$
modulo $ {\cal M}_o^{tf}(B_{0}Z)$.
Donc le noyau de l'application du lemme appartient \`a $ {\cal M}_o^{tf}(B_{0}Z)$. 

Montrons maintenant que   l'application est surjective. Soit $(f_{u})_{u\in U(q)}$ des \'el\'ements de $\ind_{B_{0}Z}(V_{0})$. On cherche un \'el\'ement $f\in \ind_{B_{0}Z}(V_{0})$ tel que $(ut)^{-1}f+ C= f_{u}+ C$ pour tout $u\in U(q)$.

Soit $\phi_{u}\in utB^{+}[1,V_{0}]$ la composante de $utf_{u}$, telle que 
$$utf_{u}\ + \  (B-ut B^{+})[1,V_{0}] \ = \ \phi_{u} \ \oplus (B-ut B^{+})[1,V_{0}] \  . $$
Par (\ref{inter}) on a 
$$utf_{u}+utC \ = \ (\phi_{u} +  utS^{+}(V_{0}))   \ \oplus \ (B-ut B^{+})[1,V_{0}] \  , $$
\end{proof}
 $$\cap_{u\in U(q)} \ (utf_{u}+ut C) \ = \ (B-B^{+}t) [1,V_{0}] \oplus_{u\in U(q)}( \phi_{u} + ut S^{+}(V_{0})) \ .$$
L'intersection n'est pas vide. Soit $f$ dans l'intersection, par exemple   $f=\sum_{u\in U(q)}\phi_{u}$. Pour tout $u\in U(q)$, on a 
$f+utC = utf_{u}+utC$, donc $(ut)^{-1}f+ C= f_{u}+ C$. Donc l'application est surjective.

   \section{$(\varphi,\Gamma)$-modules}
  
Notations comme dans le chapitre pr\'ec\'edent. On suppose de plus que  $o$ est l'anneau des entiers d'une extension finie $L$ de $\bf Q_{p}$ d'uniformisante $p_{L}$ et de corps r\'esiduel $k$.  
 On note  
$$\Lambda_{o}(U_{0})\quad :=\quad o[[U_{0}]]\quad ,  $$
 la $o$-alg\`ebre de groupe compl\'et\'ee  de $U_{0} \simeq O_{F}$, et 
 ${\cal M}s(\Lambda_{o}(U_{0}))$   la cat\'egorie  des $\Lambda_{o}(U_{0})$-modules topologiques profinis.
 La dualit\'e de Pontryagin 
  $$ M \mapsto M^{*}:=\Hom_{o}^{cont}(M,L/o) $$ est une \'equivalence de cat\'egories entre  ${\cal M}_{o-tor}(U_{0}) $ et   ${\cal M}(\Lambda_{o}(U_{0}))$
  telle que  ${M^{*}}^{*} =M$. 
  La dualit\'e de Pontryagin est un outil extr\^emement utile. Nous en voyons tout de suite des exemples.

 \begin{proposition} \label{noeth} La cat\'egorie ${\cal M}^{tf}(\Lambda_{o}(U_{0}))$  des $\Lambda_{o}(U_{0})$-modules topologiques de type fini est \'equivalente par la dualit\'e de Pontryagin \`a 
la sous-cat\'egorie des  $M\in {\cal M}_{o-tor}(U_{0})$ tels que 
 $$M^{U_{0}, p_{L}=0} \quad {\it est \  fini } \  .$$
\end{proposition}

\begin{proof} 
Par le lemme de Nakayama topologique, $M^{*}$ est un $\Lambda_{o}(U_{0})$-module de type fini si et seulement sa r\'eduction modulo l'id\'eal maximal de $\Lambda_{o}(U_{0})$ est   finie si et seulement si  $M^{U_{0}, p_{L}=0}$ est fini. \end{proof}

\begin{proposition} \label{mtf} Soit $V\in {\cal M}_{o-tor}(B)$ une repr\'esentation $Z$-localement finie et   de pr\'esentation $\ind_{B_{0}Z}^{B}(V_{0})\to V$ finie avec $V_{0}\subset V|_{B_{0}Z}$. Si $V$ est   engendr\'e  par  un sous-module $B^{+}$-stable $M$ tel que $M^{U_{0}, p_{L}=0}$ est fini, 
 alors  $D_{V}(V_{0})^{*}$ 
 est un $\Lambda(U_{0})$-module de type fini.

\end{proposition}
\begin{proof} Il existe $n \in \bf N$ tel que $t^{n}V_{0}\subset M$. La repr\'esentation $V'_{0} \in   {\cal M}_{o-tor}(B_{0}Z)$ engendr\'ee par $t^{n}V_{0}$ dans $V|_{B_{0}Z}$ est  finie, contenue dans $M$, et engendre $V$. L'espace   $<B^{+}V'_{0}>^{U_{0}, p_{L}=0}$ est fini car contenu dans $M^{U_{0}, p_{L}=0}$ qui est fini par hypoth\`ese.  Par la proposition \ref{1},  les espaces $<B^{+}V'_{0}>^{U_{0}, p_{L}=0}$ et  $D_{V}(V_{0})^{U_{0}, p_{L}=0}$ sont simultan\'ement finis ou infinis. Par la proposition \ref{noeth}, $D_{V}(V_{0})^{U_{0}, p_{L}=0}$ est fini si et seulement si  $D_{V}(V_{0})^{*}$ 
 est un $\Lambda(U_{0})$-module de type fini.
 \end{proof}
 
L'anneau de Fontaine ${\cal O}_{\cal E}$  est la compl\'etion $p$-adique du localis\'e de 
$$ \Lambda_{o}\ :=\  \Lambda_{o} (U({\bf Z_{p}})) \ .$$ par rapport au compl\'ementaire de $p_{L}\Lambda _{o}$, et  $\Lambda_{o}$ se plonge naturellement dans  ${\cal O}_{\cal E}$. 

Supposons que  $F$ est une extension finie de $\bf Q_{p}$. Alors la trace de $O_{F}  $ \`a $\bf Z_{p}$  induit un morphisme de $o$-alg\`ebres locales
$$\Lambda_{o}(U_{0} ) \ \to \ \Lambda_{o} \ \to \ {\cal O}_{\cal E} \ . $$
Le produit tensoriel ${\cal O}_{\cal E} \otimes_{\Lambda_{o}(U_{0} )} - $ va tuer le dual de Pontryagin d'une repr\'esentation $Z$-localement finie dans  ${\cal M}_{o-tor}^{tf}(B_{0}Z)$, car:

 \begin{lemma} Si $M$ est un $\Lambda_{o}(U_{0} )$-module fini, alors  ${\cal O}_{\cal E} \otimes_{\Lambda_{o}(U_{0} )}M $ est nul.
\end{lemma}
 \begin{proof} Le $o[[U({\bf Z_{p}})]]$-module $N:=o[[U({\bf Z_{p}})]]  \otimes_{\Lambda_{o}(U_{0} )}M$ est fini  et 
 $${\cal O}_{\cal E} \otimes_{\Lambda_{o}(U_{0} )}M ={\cal O}_{\cal E} \otimes_{o[[\bf Z_{p}]]}N =0 \ . $$
  \end{proof}

L'action de $T^{+}({\bf Q_{p}})$ sur $U({\bf Z_{p}})$ induit une action de $T^{+}({\bf Q_{p}})$  sur  $\Lambda $ et sur ${\cal O}_{\cal E} $. On note 
$T^{*}$ le sous-monoide de $T^{+}$ form\'e par les \'el\'ements de coefficient $1$ en bas \`a droite.

\begin{definition} Un $(\varphi, \Gamma)$-module  sur ${\cal O}_{\cal E}$ est un   ${\cal O}_{\cal E}$-module ${\cal D}$ muni d'une action semi-lin\'eaire de $T^{*}({\bf Q_{p}})$.  
Les actions de 
$$\pmatrix{p & 0 \cr 0 & 1} \ \ , \ \ \pmatrix{{\bf Z_{p} }^{*}& 0 \cr 0 & 1} \ $$
qui engendrent le monoide $T^{*}({\bf Q_{p}})$,  sont not\'ees $\varphi$ et $\Gamma$.
Un  $(\varphi, \Gamma)$-module est dit \'etale si $\varphi$ est \'etale.
\end{definition}

Soit  $V\in {\cal M}_{o-tor}(G)$ de pr\'esentation  $\ind_{B_{0}Z}^{B}(V_{0}) \to V$ avec $V_{0}\subset V|_{B_{0}Z}$. 
 Le  dual de Pontryagin  $D_{V}(V_{0})^{*}$ de  $D_{V}(V_{0})$ est un $\Lambda (U_{0})$-module muni d'une action semi-lin\'eaire continue de $T^{+}=T^{+}(F)$.
Par restriction \`a   $T^{*}({\bf Q_{p}})$,  les modules
$$\Lambda  \otimes_{\Lambda(U_{0} )}D_{V}(V_{0})^{*} \ \ , \ \  {\cal O}_{\cal E} \otimes_{\Lambda(U_{0} )}D_{V}(V_{0})^{*}$$
sont des $(\varphi, \Gamma)$-modules sur $\Lambda, \    {\cal O}_{\cal E} $.
 
 Nous avons d\'emontr\'e le th\'eor\`eme suivant:

\begin{theorem} \label{th3} On suppose que $F$ est une extension finie de $\bf Q_{p}$. Soit $V\in {\cal M}_{o-tor}(B)$ une repr\'esentation $Z$-localement finie et de  pr\'esentation finie $\ind_{B_{0}Z}^{B}(V_{0})\to V$ avec $V_{0}\subset  V|_{B_{0}Z}$.  Alors, le $(\varphi, \Gamma)$-module sur ${\cal O}_{\cal E} $
$${\cal D}(V):=  {\cal O}_{\cal E} \otimes_{\Lambda(U_{0} )}D_{V}(V_{0})^{*}$$
est \'etale et ne d\'epend pas du choix de la pr\'esentation finie. C'est  un ${\cal O}_{\cal E} $-module de type fini si $V$ est engendr\'e  par  un sous-module $B^{+}$-stable $M$ tel que $M^{U_{0}, p_{L}=0}$ est fini.
\end{theorem}

\section{Repr\'esentations de $G:=GL(2,F)$}

Nous nous restreignons  maintenant  aux repr\'esentations lisses de $B$ qui se prolongent \`a $G$. 
C'est une information profonde qui permet d'aller plus loin. Nous allons g\'en\'eraliser les r\'esultats obtenus par  Hu \cite{Hu} pour des repr\'esentations irr\'eductibles lisses de $G$ sur $k$ \`a des repr\'esentations lisses de torsion de $G$ sur $o$.
Il n'y a pas d'hypoth\`ese sur la caract\'eristique de $F$, mais lorsque $F$ est de caract\'eristique $0$ 
 nous en d\'eduirons  des repr\'esentations satisfaisant la condition
du th\'eor\`eme \ref{th3}.  Les techniques d\'evelopp\'ees pour les repr\'esentations de $B$ dans les chapitres pr\'ec\'edents joints \`a celles de Hu permettent d'obtenir des r\'esultats nouveaux sur les repr\'esentations lisses de $G$  sur $o$ (non necessairement de torsion).

On note  $K:=GL(2,O_{F} ) $.

  Soit $V\in {\cal M}_{o}(G)$ et soit  $V_{0}\in {\cal M}_{o}(KZ)$.  Si $V$ est un quotient de  $\ind_{KZ}^{G}(V_{0})$, on dit que 
  $$ \ind_{KZ}^{G}(V_{0}) \to V $$
  est une pr\'esentation de $V$. On dit que la pr\'esentation est finie si $V_{0}$ est de type fini et si 
    le noyau  $R_{V}(V_{0}) \in  {\cal M}_{o}(G)$  de  la pr\'esentation est  de type fini.

  \begin{proposition}    Soit $V\in {\cal M}_{o}(G)$. Alors $V$ est de pr\'esentation finie si et seulement si $V|_{B}$ est de pr\'esentation finie.
  \end{proposition}
  
\begin{proof} a) Pr\'eliminaires. La d\'ecomposition de Bruhat $G=BK, \ B\cap KZ=B_{0}Z$, montre que la restriction \`a $B$ d'une repr\'esentation induite compacte $ \ind_{KZ}^{G}(V_{0}) $ est 
\'egale \`a  $ \ind_{B_{0}Z}^{B}(V_{0}) $. Comme les repr\'esentations sont lisses, une repr\'esentation $V\in {\cal M}_{o}(G)$ est de type fini si et seulement si sa restriction \`a $B$ est de type fini, et une repr\'esentation $V_{0}\in {\cal M}_{o}(KZ)$ ou 
$V_{0}\in {\cal M}_{o}(B_{0}Z)$ est de type fini si et seulement si sa restriction \`a $Z$ est de type fini.

b) Soit  $ \ind_{KZ}^{G}(V_{0}) \to V$  une pr\'esentation finie de $V$ de noyau $R_{V}(V_{0})$. Alors  sa restriction \`a $B$  est une pr\'esentation finie
 $ \ind_{B_{0}Z}^{B}(V_{0}) \to V$ de $V|_{B}$, par les pr\'eliminaires.
 
c) Supposons que $V|_{B} $ soit de pr\'esentation finie. On choisit une pr\'esentation finie $ \ind_{B_{0}Z}^{B}(V_{0}) \to V$ de $V|_{B}$ avec $V_{0}\subset V|_{B_{0}Z}$. 
La sous-repr\'esentation $$V'_{0}:= <K V_{0}> \ \in {\cal M}_{o}(KZ)$$
de $V|_{KZ}$ est de type fini et la pr\'esentation naturelle
 $$ \ind_{B_{0}Z}^{B}(V'_{0}) \to V$$ 
 de $V|_{B}$  est finie par la proposition \ref{fp}.  Cette pr\'esentation finie de $V|_{B}$ est la  restriction \`a $B$  de la pr\'esentation naturelle 
  $$ \ind_{KZ}^{G}(V'_{0}) \to V$$ 
  de $V$. Cette pr\'esentation est finie par les pr\'eliminaires.
\end{proof}

 \bigskip  Rappelons les r\'esultats de Hu \cite{Hu}.
 
 \begin{theorem} \label{Hu} Lorsque $V\in {\cal M}_{o}(G)$ et  $V_{0}\in {\cal M}_{o}(KZ)$ sont irr\'eductibles avec un caract\`ere central et $V_{0}\subset V|_{KZ}$, on a les propri\'et\'es suivantes:

 a) L'espace vectoriel $< B^{+}V_{0}>$ ne d\'epend pas du choix de $V_{0}$.
 
 b)   $\Delta _{V}(V_{0})$ contient $V^{I_{1}}$, avec \'egalit\'e si $V$ n'est pas supersinguli\`ere.
 
 c)   $\Delta _{V}(V_{0})$ est finie si et seulement si $R_{V}(V_{0})$ est finie, et  dans ce cas $V$ est admissible.
 
 d)  Un quotient admissible $V'$ de $\ind_{KZ}^{G}(V_{0})$ tel que $\Delta_{V'}(V_{0})$ est fini,  est de longueur finie. De plus,  $< B^{+}V_{0}>^{U_{0}} \subset \Delta _{V}(V_{0})$, et $N_{V'}(V_{0})\in {\cal M}_{o}(G)$ est  de type fini.
 
 e) L'inclusion   $\Delta_{V}\ \subset \ <K \Delta_{V}>$ est le ``diagramme canonique'' de $V$, et deux repr\'esentations irr\'eductibles $ {\cal M}_{o}(G)$ avec un caract\`ere central sont isomorphes si et seulement leurs diagrammes canoniques sont isomorphes.
  
  \end{theorem}
  
  \begin{proof} a) (\cite{Hu} Cor. 3.15) .
  
  b)  (\cite{Hu} Th. 1.2) . 
  
  c)  (\cite{Hu} Th. 1.3, Th. 4.3) . De plus, $\Delta _{V}(V_{0})$ est finie si et seulement si $R_{V}(V_{0})$ est finie, pour tout quotient non trivial $V$ de $\ind_{KZ}^{G}(V_{0})$.
  
  d) (\cite{Hu} Cor. 3.24, Prop. 4.5) .
  
  e) (\cite{Hu} Th. 3.17) .
  \end{proof}

\bigskip Nous allons \'etendre ces r\'esultats au cas o\`u $V\in {\cal M}_{o-tor}(G)$ 
et $V_{0} \in {\cal M}_{o-tor}(KZ)$ ne sont pas irr\'eductibles. 
Notons $$ {\cal M}_{o-tor}^{apf}(G)$$ la sous-cat\'egorie de  $ {\cal M}_{o-tor}(G)$ des repr\'esentations admissibles et de pr\'esentation finie.  Nous allons voir  que  $ {\cal M}_{o-tor}^{apf}(G)$ est une cat\'egorie abelienne, et une sous-cat\'egorie de Serre de $  {\cal M}_{o-tor}(G)$ contenue dans la cat\'egorie $ {\cal M}_{o-tor}^{tf}(G)$ des repr\'esentations de type fini.

 \begin{theorem} \label{lf} Soit  $V\in {\cal M}_{o-tor}(G)$ une repr\'esentation  admissible de pr\'esentation finie. Alors $V$ est de longueur finie et   tous les sous-quotients de $V$ sont admissibles, de pr\'esentation finie. 
    \end{theorem}

\begin{proof}  Notons que$V$ est annul\'e par une puissance de $p_{L} $.  Par un th\'eor\`eme g\'en\'eral, \cite{Lang} th. 2, tout sous-quotient d'une repr\'esentation admissible  $V\in {\cal M}_{o-tor}(G)$  annul\'e par une puissance de $p_{L}$ est admissible. 

Soit $V_{0}\in {\cal M}_{o-tor}(KZ)$  contenue dans $V|_{KZ}$ telle que  le  morphisme naturel $f: \ind_{KZ}^{G}V_{0}\to V$ soit une pr\'esentation finie de $V$. Alors  $\Delta_{V}(V_{0})$ est fini (Prop. \ref{1}).

Montrons que $V$ est de longueur finie.

a)  Si $V_{0}$ est  irr\'eductible avec un caract\`ere central,  un quotient admissible de $\ind_{KZ}^{G}(V_{0})$ est  de longueur finie par Hu (Th. \ref{Hu} d)). 

b) Si   tous les sous-quotients irr\'eductibles  de $V_{0}$ ont un caract\`ere central, on montre par induction sur la longueur de $V$ que 
  $V$ est de longueur finie. Pour une sous-repr\'esentation  irr\'eductible $V'_{0}$ de $V_{0}$  de quotient  $V_{0}''=V_{0}/V_{0}'$. 
 l'image de la suite exact $$0 \to  \ind_{KZ}^{G}(V_{0}')\to  \ind_{KZ}^{G}(V_{0}) \to   \ind_{KZ}^{G}(V''_{0}) \to 0 \ $$
  par la pr\'esentation finie $ f:\ind_{KZ}^{G}(V_{0}) \to V $ est une suite exacte
$$0 \to  V' \ \to \  V \  \to  \ V'' \  \to 0 \  .$$
Les noyaux forment une suite exacte 
$$0 \to  \ R_{V'}(V'_{0}) \ \to  \ R_{V}(V_{0}) \ \to  \ R_{V''}(V''_{0}) \to 0 \ .$$
Comme $R_{V}(V_{0})\in {\cal M}_{o-tor}(G)$ est de type fini,   le quotient $R_{V''}(V''_{0})$ est de type fini.  Par la proposition \ref{1} , $\Delta_{V}(V_{0})$ est fini, donc  
   $\Delta_{V'}(V'_{0})$ \'etant contenu dans  $\Delta_{V}(V_{0})$ est fini.
    Par  le th\'eor\`eme \ref{Hu} c), $R_{V'}(V'_{0})$ est de type fini. Les repr\'esentations $V',V''$ sont donc de pr\'esentation finies. Elles sont aussi admissibles. Par a), $V'$ est de longueur finie.  Par induction sur la longueur de $V_{0}$,
  $V''$ est de longueur finie. Donc $V$ est de longueur finie.

c) Il existe une extension finie $o'/o$ telle que    tous les sous-quotients irr\'eductibles  de $V_{0,o'}:=o'\otimes_{o}V_{0}$ ont un caract\`ere central. Par  le lemme \ref{ext},   $V_{o'}:=o'\otimes_{o}V$ est admissible et 
 la  pr\'esentation $\ind_{KZ}^{G}(V_{0,o'})\to V_{o'}$ est   finie. Par b),  $V_{o'}=o'\otimes_{o}V\in {\cal M}_{o'}(G)$ est de longueur finie. A fortiori $V$ est de longueur finie.

Montrons maintenant que tous les sous-quotients de $V$ sont de pr\'esentation finie. Tous les sous-quotients de $V$ d'une repr\'esentation de longueur finie sont de longueur finie. Par la proposition \ref{suite} b), tous les quotients de $V$ sont de pr\'esentation finie. 
Montrons qu'une sous-repr\'esentation $V'$ de $V$  est  de pr\'esentation finie. Soit $V'':=V/V'$.

On se ram\`ene par extension des scalaires, comme on le peut par  le th\'eor\`eme \ref{lf} 3) et le lemme \ref{ext},  au cas 
o\`u tous les sous-quotients irr\'eductibles de $V$ ont un caract\`ere central. 

On choisit, comme on le peut (Prop. \ref{suite} a)), des $o[KZ]$-sous-modules $V_{0}', V_{0},V''_{0}$ finis et engendrant   $V',V,V''$ formant une suite exacte
 $$0\to V' _{0}\to V_{0} \to V''_{0} \to 0$$
dans  $ {\cal M}_{o}^{f}(KZ)$.  L'injection  $\Delta_{V'}(V'_{0})\to \Delta_{V}(V_{0})$ implique que $\Delta_{V'}(V'_{0})$ est fini puisque $ \Delta_{V}(V_{0})$ est fini.

Supposons  $V'$ irr\'eductible.  On choisit une sous-repr\'esentation irr\'eductible $W_{0}$ de $V'_{0} $. Naturellement $\Delta_{V'}(W_{0})$ est fini puisque contenu dans $\Delta_{V'}(V'_{0})$. 
Par  le th\'eor\`eme \ref{Hu} c), la pr\'esentation $\ind_{KZ} ^{G}(W_{0})\to V'$ est finie. 

Supposons $V'$ r\'eductible.  Par une r\'ecurrence  sur  la longueur, $V'/W$  est de pr\'esentation finie pour toute   sous-repr\'esentation irr\'eductible $W$ de $V'$, car elle est contenue dans $V/W$ qui est admissible et de longueur finie. Comme $W, V'/W$ sont de pr\'esentation finie, la repr\'esentation $V'$ est de pr\'esentation finie par la proposition \ref{suite} a). 
  \end{proof}  

  \begin{proposition} Soit $V\in {\cal M}_{o-tor}(G)$   de type fini, et soit $V_{0}$ une sous-repr\'esentation de $V|_{GL(2,O_{F})Z}$ g\'en\'eratrice. Alors  $R_{V}(V_{0})\in  {\cal M}_{o-tor}(G)$  est   de type fini, si $V$ est admissible et  $\Delta_{V}(V_{0})$  fini. 
  Inversement, $R_{V}(V_{0} )$ de type fini implique $\Delta_{V}(V_{0})$ fini.   \end{proposition}

\begin{proof}  Vu la proposition \ref{1}, il suffit de montrer que   $\Delta_{V}(V_{0})$ fini et $V$ admissible  implique $R_{V}(V_{0} )$  de type fini. Comme pr\'ec\'edemment, on se ram\`ene au cas o\`u les sous-quotients irr\'eductibles de $V_{0}$ ont un caract\`ere central.  Une v\'erification facile permet ensuite de se ramener au cas o\`u $V_{0}$ est irr\'eductible. Puis on applique Hu.

\end{proof}

On note 
 $$s:=\pmatrix {0&1 \cr 1&0} \ \ , \ \ \w:=\pmatrix {0&1 \cr p_{F}&0}  = st  \  .$$
 L'\'el\'ement $\w$ normalise le sous-groupe d'Iwahori $I$  image inverse de $B(q):=B({\bf F_{q}})$ dans $K$, ainsi que 
 le pro-$p$-sous-groupe de Sylow $I_{1}$ de $I$. Le normalisateur de $I$ dans $G$ est le groupe $ IZ \cup \w IZ$.  Le groupe $G$ est engendr\'e par $KZ$ et $\w$.

\begin{lemma}\label{base} On a  $(B-B^{+})K \ = \ \w B^{+} K$ et  l'union 
$$G\quad =\quad B^{+}K \ \cup \ \w B^{+} K$$ est disjointe et $I$-stable \`a gauche.
\end{lemma} 
 
\begin{proof} La d\'ecomposition de Bruhat, $G=BK$ et  $B\cap K=B_{0}$ implique que 
$$G= B^{+}K \cup (B-B^{+})K \ $$
et que l'union est disjointe.
Comme l'on a $$B^{+}\ =\ \cup_{n\in \bf N}\ B_{0}t^{n}Z \quad , \quad I\ =\ B_{0}U_{1}^{-} \quad , $$ o\`u $U_{1}^{-}= U^{-}(p_{F}O_{F}) $ et $U^{-} $ est le groupe des matrices strictement triangulaires inf\'erieures de $G$, nous avons 
$$It^{n}KZ = B_{0} t^{n}KZ=B_{0}t^{n}KZ \quad, \quad B^{+}K= \cup_{n\in \bf N} It^{n}KZ \ .$$
Donc $B^{+}K $ est $I$-stable \`a gauche.
   Par les d\'ecompositions de Cartan (v\'erifier la derni\`ere \'egalit\'e),  $G$ est \'egal aux unions disjointes: 
$$G \ = \ \cup_{n\in \bf N}K t^{n} KZ  \ = \ \cup_{n\in \bf N}I t^{n} KZ  \ \cup_{n\in \bf N}\w I t^{n} KZ \ . $$
On en  d\'eduit avec la formule pour $B^{+} K$,  que le compl\'ementaire de $B^{+}K$ dans $G$ est $\w B^{+}K$ qui est  aussi $I$-stable car $\w$ normalise $I$.
\end{proof}

 Comme exemple d'application,  on  a:

\begin{proposition} \label{4} Soit $V\in {\cal M}_{o}(G)$   et soit $V_{0}$  une sous-repr\'esentation de $V|_{KZ}$ g\'en\'eratrice. Alors les deux $o$-modules
$$<B^{+}V_{0}> \quad , \quad <(B-B^{+})V_{0}> \ = \  \w < B^{+}V_{0}> \ , $$
 sont stables par $IZ$. Leur intersection
$$ \Delta_{V} (V_{0}) =<B^{+}V_{0}> \ \cap \ \w <B^{+}V_{0}>$$
contient $V_{0}^{I_{1}}$ et est stable par  $ IZ \cup \w IZ$.  
\end{proposition}

\begin{proof} Evident avec le lemme  \ref{base}.  Comme $\w$ normalise $I_{1}$, le sous-espace $V_{0}^{I_{1}}$ de $< B^{+}V_{0}>$ est contenu dans  $\w < B^{+}V_{0}>$.
 
  \end{proof}

 \begin{proposition} \label{4'} Soit $V\in {\cal M}_{o-tor}(G)$   et soit $V_{0}$  une sous-repr\'esentation de $V|_{KZ}$ g\'en\'eratrice. Alors $\Delta_{V} (V_{0})\neq 0$.

 \end{proposition}
 \begin{proof}  Par la proposition  \ref{4}, il suffit de voir que si $V_{0}\in {\cal M}_{o-tor}(K)$  est non nul, alors $V_{0}^{p_{L}=0, I_{1}}\neq 0$.Soit $v\in V_{0}$ non nul, et soit $n \in \bf N$ minimal tel que $p_{L}^{n}v=0$. On a $n\neq 0$ donc $V_{0}$ contient un \'el\'ement non nul $w:=p_{L}^{n-1}v$ annul\'e par $p_{L}$. 
 La sous-repr\'esentation de $V_{0}|_{I_{1}}$ engendr\'ee par $w$  est non nulle et appartient \`a ${\cal M}_{k}(I_{1})$; elle a un vecteur non nul fixe par $I_{1}$ qui est un pro-$p$-groupe.
    \end{proof}

 Soit $V\in {\cal M}_{o}(G)$ et soit $V_{0}\in  {\cal M}_{o}(KZ)$ avec $V_{0}\subset V|_{KZ}$ g\'en\'eratrice. Suivant Hu \cite{Hu }Lemme 3.5,  $ \Delta_{V}(V_{0})$ est le premier terme (pour $n=0$) d'une suite  $( \Delta_{V}^{(n)}(V_{0}))_{n\in \bf N}$ de repr\'esentations de $IZ$ contenues dans  $<B^{+}V_{0}>|_{IZ}$, d\'efinies par 
r\'ecurrence
$$ \Delta_{V}^{(n)}(V_{0}):=<B^{+}V_{0}> \  \ \cap  \ <K \w \Delta_{V}^{(n-1)}(V_{0})> $$
pour $n\geq 1$.
Comme $<B^{+}V_{0}>$ est stable par $IZ$, les $ \Delta_{V}^{(n)}(V_{0})$ sont stables par $IZ$.

\begin{lemma} \label{suite2}  Soit $V\in {\cal M}_{o}(G)$ et soit $V_{0}\in  {\cal M}_{o}(KZ)$ avec $V_{0}\subset V|_{KZ}$ g\'en\'eratrice.  La suite $( \Delta_{V}(V_{0})^{(n)})_{n\in \bf N}$ est croissante, et
$$ \Delta_{V}^{(n)}(V_{0})=\Delta_{V}(V_{0}) + \sum_{u\in U(q)} ut  \Delta_{V}^{(n-1)}(V_{0}) \ . $$
 L'union des  $( \Delta_{V}(V_{0})^{(n)})_{n\in \bf N}$ est \'egale \`a $<B^{+}V_{0}>$ si  $ \Delta_{V}(V_{0})$ engendre $V$.
\end{lemma} 

\begin{proof} 

a) Un syst\`eme de repr\'esentants de  $K/I$ est $\{1,U(q)s\}$, donc un  syst\`eme de repr\'esentants de $Kw/I$ est  $\{\w,  U(q)t\}$. Pour toute sous-repr\'esentation $M$ de $V|_{I}$,  on a
 $$<K \w M> =\w M +  \sum_{u\in U(q)} ut M \ . $$
 Prenant $M=  \Delta_{V}(V_{0})$ qui est stable par $\w$ on obtient 
 $$<K \w \Delta_{V}(V_{0})>  \ = \   \Delta_{V}(V_{0})+  \sum_{u\in U(q)} ut  \Delta_{V}(V_{0}) \ . $$
On a donc $ \Delta_{V}^{(1)}(V_{0}) = <K \w \Delta_{V}(V_{0})>$  car cette repr\'esentation est   contenue dans $<B^{+} V_{0}>$. 
On voit que $\Delta_{V}(V_{0})\subset  \Delta_{V}^{(1)}(V_{0})$, puis 
par induction sur $n$, que 
 $   \Delta_{V}^{(n-1)}(V_{0})\subset    \Delta_{V}^{(n)}(V_{0})$ pour tout $n\geq 1$.
 Lorsque $M$ est contenu dans $<B^{+}V_{0}>$, on voit que 
 $$<B^{+}V_{0}>\cap <K \w M> \ =\ \w (<\w B^{+}V_{0}>\ \cap \  M)\ +   \  \sum_{u\in U(q)} ut M \ . $$
 Si $\Delta_{V}(V_{0}) \subset M$ le premier terme du membre de droite est $\w \Delta_{V}(V_{0})= \Delta_{V}(V_{0})$.
 Prenant $M=   \Delta_{V}^{(n-1)}(V_{0}) $  on obtient 
$$ \Delta_{V}^{(n)}(V_{0})=\Delta_{V}(V_{0}) + \sum_{u\in U(q)} ut  \Delta_{V}^{(n-1)}(V_{0}) \  $$
ce qui entraine
  $$<K \w \Delta_{V}^{(n-1)}(V_{0})>  \ = \ \w \Delta_{V}^{(n-1)}(V_{0}) + \Delta_{V}^{(n)} (V_{0}) \  . $$
Donc la filtration croissante
$$( <K\w \Delta_{V}^{(n)}(V_{0}) >)_{n\in {\bf N}}   $$
est  contenue dans la sous-repr\'esentation $V'$ de $V$ engendr\'ee par    $\Delta_{V}(V_{0})$, et sa limite   est stable  par $K$ et $\w$. Comme $G$ est engendr\'e par $K$ et $\w$,  cette limite est  $V'$. L'union des intersections avec  $<B^{+}V_{0}>$ est $V' \ \cap  \ <B^{+}V_{0}>$.  L'union des  $( \Delta_{V}(V_{0})^{(n)})_{n\in \bf N}$ est donc  \'egale \`a $V' \ \cap  \ <B^{+}V_{0}>$.

\end{proof}

\begin{remark} Si $V_{0}\in {\cal M}_{o}(KZ)$ est engendr\'ee par  $V_{0}^{I_{1}}$, alors  $V$ est engendr\'e par $\Delta_{V}(V_{0})$  qui  contient  $V_{0}^{I_{1}}$ (lemme \ref{4}). C'est toujours le cas si $V_{0}$ est irr\'eductible.
\end{remark}

  \begin{lemma}  \label{6}  Soit  $V\in {\cal M}_{o}(G)$  de type fini (resp. de pr\'esentation finie) et soit $M\in  {\cal M}_{o}(B^{+})$ de type fini contenu dans $V|_{B^{+}}$. Alors il existe une presentation (resp. finie)
  $\ind_{KZ}^{G}(V_{0})\to V$ de $V$ avec $V_{0}\subset V|_{KZ}$ de type fini,  $\Delta_{V} (V_{0})$ engendre $V$ et $M\ \subset \ <B^{+}V_{0}>$.

\end{lemma}
  
   \begin{proof}  Si $V$ est de presentation finie, on choisit, comme on le peut,
    une pr\'esentation finie $\ind_{KZ}^{G}(V'_{0})\to V$
  de $V$ avec $V'_{0}\subset V|_{KZ}$ contenant un syst\`eme fini de g\'en\'erateurs de $M$ (proposition \ref{cor1}). Dans $\ind_{KZ}^{G}(V'_{0})$,
     $$W_{0} \ := \ <K(\w [1,V'_{0}] + [1,V'_{0}]) >$$est un $o[KZ]$-sous-module de type fini; son image $V_{0}$ dans $V$ contient $V'_{0} \cup \w V'_{0}$. Par la proposition \ref{fp}, $$\ind_{KZ}^{G}(V_{0} )\to V$$ est une pr\'esentation finie de $V$, qui v\'erifie les  propri\'et\'es: 
     
    a) $\Delta_{V}(V_{0})$  engendre $V$, car  $\Delta_{V}(V_{0})$  contient
   $V'_{0} $.

b) $M\  \subset \   <B^{+} V_{0}>$.

Si $V$ est simplement de type fini, c'est le m\^eme argument en supprimant ``finie'' dans pr\'esentation finie, et en ajoutant $V_{0}$ de type fini. 
   
 \end{proof}

 \begin{lemma} \label{rel}    Soit $V\in {\cal M}_{o}(G)$ et $V_{0}\in {\cal M}_{o}(KZ)$  avec $V_{0}\subset V|_{KZ}$ et g\'en\'erateur. Alors la relation 
$$\w v_{0}+ \sum_{u\in U(q)} ut v_{u}  = 0$$
 entre $q+1$ \'el\'ements $v_{0}, (v_{u})_{u\in U(q)}$ de 
$<B^{+}V_{0}>$ implique que ces \'el\'ements appartiennent \`a   $\Delta_{V}(V_{0})$.
\end{lemma} 

\begin{proof} Comparer avec  \cite{Hu} Lemme 3.1.  a) Comme $$ <\w B^{+} V_{0}>\cap <U(q)s \w B^{+} V_{0}> \ = \  <\w B^{+} V_{0}>\cap <B^{+}t V_{0}> $$  est contenue dans $\Delta_{V}(V_{0})$, la relation implique $\w v_{0}\in \Delta_{V}(V_{0})$. Comme $\Delta_{V}(V_{0})=\w\Delta_{V}(V_{0})$ on en d\'eduit $v_{0}\in \Delta_{V}(V_{0})$. 

b) Comme $1, U(q)s$ est un syst\`eme de repr\'esentants de $K/I$, tout \'el\'ement $f$ de $ \ind_{I}^{K}(<\w B^{+} V_{0}>)$ s'\'ecrit   $[1,\w v_{0}]+ \sum_{u\in U(q)}us [1,\w v_{u}]$ pour $v_{0}, (v_{u})_{u\in U(q)}$ uniques de 
$<B^{+}V_{0}>$. Clairement $f$ appartient au
 noyau $R$ de l'application naturelle $K$-\'equivariante
$$ \ind_{I}^{K}(<\w B^{+} V_{0}>) \ \to \  V \   $$
si et seulement si la relation du lemme est v\'erifi\'ee.  Comme $R$ est $K$-\'equivariant, si $f\in R$ alors
il existe un \'el\'ement $f_{u}\in R$ tel que $f_{u}(1)=\w v_{u}$ pour tout $u\in U(q)$.
On d\'eduit de a) que $v_{u}\in \Delta_{V}(V_{0})$ pour tout $u\in U(q)$.

\end{proof}

\begin{remark} Notons que  $$V\quad =\quad <K\w B^{+} V_{0}> \ + \ V_{0} \quad  .$$
\end{remark}
\begin{proof} $B^{+}K$ \'etant $I$-stable \`a gauche,
nous avons $$K\w B^{+}K = \w B^{+}K \cup U(q)tB^{+}K=  \w B^{+}K \cup B^{+}t K=G-B_{0}Z \ . 
 $$
 \end{proof}

\bigskip On consid\`ere la fonction
$$\ell_{V_{0}}:\ <B^{+}V_{0}> \ \to \ {\bf N} \cup \infty \ .$$
associ\'ee \`a a filtration croissante $(\Delta_{V}^{(n)}(V_{0}))_{n\in bf N}$. Si $v$ n'appartient \`a aucun 
$\Delta_{V}^{(n)}(V_{0})$ alors $\ell_{V_{0}}(v)=\infty$. Sinon, 
 $\ell_{V_{0}}(v)$  est le plus petit entier $n\in \bf N$ tel que $v\in   \Delta_{V}^{(n)}(V_{0})$. 

\bigskip Pour tout entier  $m\in \bf N$, notons $U_{m}:= U(p_{F}^{m}O_{F})$ et $U^{-}_{m}:= U^{-}(p_{F}^{m}O_{F})$.

 \begin{lemma}\label{Am}  Soit $V\in {\cal M}_{o}(G)$ de pr\'esentation $\ind_{KZ}^{G}(V_{0}) \to V$ avec $V_{0}\subset V|_{KZ}$.  Soit  $v \ \in  \ (\cup_{n\geq 1}\Delta_{V}^{(n)}(V_{0})) \ - \ \Delta_{V}(V_{0})$. Alors  pour tout $m\geq 1$, 
$$\ell_{V_{0}}(A_{m}v) = \ell_{V_{0}}(v)+m   \quad, \quad 
 A_{m} \ := \ \sum_{u\in U_{0}/U_{m}}u t^{m} \ . $$

\end{lemma} 

\begin{proof} Comparer avec\cite{Hu} lemme 3.8.  

a) Le cas $m=1$.  Posons   $\ell_{V_{0}}(v)=n \geq 1$. Par  le lemme \ref{suite2}, $A _{1}v\in  \Delta_{V}^{(n+1)}(V_{0})$. Supposons  $A_{1}v \in \Delta_{V}^{(n)}(V_{0})$. 
On \'ecrit
$$A_{1}v = \w v_{0} + \sum_{u\in U(q)}ut v_{u}$$
avec des \'el\'ements $v_{0}\in   \Delta_{V}(V_{0}) =\w \Delta_{V}(V_{0})$ et $( v_{u})_{u\in U(q)}$ dans $ \Delta_{V}^{(n-1)}(V_{0})$.
Les \'el\'ements $v,v_{0},( v_{u})_{u\in U(q)}$ sont li\'es par la relation
$$0= \w v_{0} +   \sum_{u\in U(q)}ut( v_{u}-v) \ . $$
Par le lemme \ref{rel},  les \'el\'ements $( v_{u}-v)_{u\in U(q)}$ appartiennent \`a $\Delta_{V}(V_{0})$. Donc $v\in \Delta_{V}^{(n-1)}(V_{0})$ ce qui contredit l'hypoth\`ese $\ell_{V_{0}}(v)=n$.

b) Le cas  $m>1$.  Le $m$-i\`eme it\'er\'e $A_{1}\circ \ldots \circ A_{1}$ de $A_{1}$ est \'egal \`a $A_{m}$.
\end{proof}

\begin{lemma} Soient $a\in \bf N$  non nul et soit $g\in U_{a}^{-}$. Alors 
l'application 
 $$u \mapsto u_{g} \quad : \quad U_{0} \to U_{0} \quad , \quad gu \in  u_{g} T_{0} U_{a}^{-}$$
est bien d\'efinie,  et induit par passage au quotient une bijection de $U_{0}/U_{b}$ pour tout $b\in \bf N$.
\end{lemma}

\begin{proof} Le
 groupe $C_{a,b} := <U_{a}^{-}, U_{b}>$ engendr\'e par $U_{a}^{-}$ et $U_{b}$ a une  d\'ecomposition d'Iwahori  (c'est facile \`a v\'erifier) : 
$$C_{a,b}\quad = \quad U_{a}^{-}\ (T_{0}\cap C_{a,b})\ U_{b}\quad = \quad U_{b}\ (T_{0}\cap C_{a,b})\ U_{a}^{-} \quad .$$
Pour $g\in U^{-}_{a}$ et 
$u\in U_{0} $, l'\'el\'ement $gu\in C_{a,0}$ s'\'ecrit uniquement $gu = u_{g}hu'$ avec $u_{g}\in U_{0}, \ h\in T_{0} ,\ u'\in U^{-}_{a}$. Donc l'application $u\mapsto u_{g}$ est bien d\'efinie. Si $u\in U_{b}$
l'\'el\'ement $gu$ appartient \`a $ C_{a,b}$, alors $u_{g}\in U_{b}$. Inversement si $u_{g}\in U_{b}$
alors $gu \in U_{b}T_{0} U^{-}_{a}$. L'\'egalit\'e dans $M(2,O_{F})$
$$\pmatrix {1&0 \cr y_{a }&1} \pmatrix {1&y_{0} \cr 0&1}  \quad = \quad \pmatrix {1&x_{b} \cr 0&1} 
\pmatrix {x &0 \cr 0&y} \pmatrix {1&0 \cr x_{a }&1}$$
implique $y x_{b}=y_{0}$. Si $y\in O_{F}^{*},  x_{b}\in p_{F}^{b}O_{F},$ alors  $y_{0}\in p_{F}^{b}O_{F}$. Donc  $u_{g}\in U_{b}$  implique $u\in U_{b}$.
\end{proof}

  \begin{proposition} \label{5} Soit  $V\in {\cal M}_{o}(G)$ de type fini et soit
 $V_{0}\in {\cal M}_{o}(KZ)$ de type fini  contenue dans $V|_{KZ}$ g\'en\'erateur et telle que
 $\Delta_{V}(V_{0})$ engendre $V$. Alors  tout $v\in <B^{+} V_{0}>^{U_{0}} $ engendrant 
 une sous-repr\'esentation admissible $V(v)$ de $V$, appartient \`a $ \Delta_{V}(V_{0}).$

 \end{proposition}
\begin{proof} Voir \cite{Hu} Proposition 4.4 (ii). 

Soit $m$ un entier $\geq 1$.  Gardons les notations du lemme pr\'ec\'edent et de sa preuve. On choisit, comme on le peut, un entier $a\geq 1$ tel que  le groupe $C_{a,0}$ fixe $v$.
Montrons que $C_{a,0}$ fixe aussi $A_{m}v$.
 Pour $g\in U_{0}$, la multiplication \`a gauche par $g$ dans $U_{0}$ induit une bijection de $U_{0}/U_{m}$, donc 
 $gA_{m}v= g \sum_{u\in U_{0}/U_{m}}u t^{m} v = A_{m}v$.
Pour $g \in U_{a}^{-}$,  on a 
$$gA_{m}v= \sum_{u\in U_{0}/U_{m}}u_{g} t^{m} v \quad = \quad A_{m}v \quad ,$$
par le lemme pr\'ec\'edent.
Le $o$-module de type fini $V(v)^{C_{a,0}}$ contient $A_{m}v$ pour tout entier naturel non nul $m$.
Les \'el\'ements $(A_{m}v)_{m\geq m_{v}}$ ne peuvent pas \^etre lin\'eairement ind\'ependants sur $o$.
Le lemme \ref{Am}  implique $v\in  \Delta_{V}(V_{0}).$
\end{proof}

 \begin{theorem} \label{aut} Soit  $V\in {\cal M}_{o-tor}(G)$ admissible de pr\'esentation finie. Alors 
 $M^{U_{0}} $ est fini, pour tout $M\in {\cal M}_{o-tor}(B^{+})$ de type fini contenu dans $V|_{B^{+}}$.
 \end{theorem} 

 \begin{proof}  Proposition \ref{5}, Lemme \ref{6}, Th\'eor\`eme \ref{tech} b).

 \end{proof} 
 
 \begin{corollary}  \label{aut} $D_{V}(V_{0})^{*}$ est un $\Lambda_{o}(U_{0})$-module de type fini pour tout  $V_{0}\in {\cal M}_{o}(B_{0}Z)$ de type fini contenu dans $V|_{B_{0}Z}$ et g\'en\'erateur.
  \end{corollary} 
 \begin{proof} Proposition \ref{mtf}.
 \end{proof} 
 
\begin{corollary} L'ensemble ${\cal P}^{+}(V)$ des $M \in  {\cal M}_{o-tor}(B^{+})$ contenus dans $V|_{B^{+}}$   et engendrant $V$ a un unique \'el\'ement minimal $V^{+}$.
\end{corollary} 

 \begin{proof} Comparer avec \cite{SVig} Proposition 9.9.  Comme $\Lambda_{o}(U_{0})$ est noetherien,  par dualit\'e de Pontryagin, toute suite d\'ecroissante dans ${\cal P}^{+}(V)$ est stationnaire.
 L'ensemble  ${\cal P}^{+}(V)$ est stable par intersection \cite{SVig} Lemma 2.2. Donc l'intersection $V^{+}$ des \'el\'ements de ${\cal P}^{+}(V)$ appartient \`a  ${\cal P}^{+}(V)$. 
 \end{proof}

\begin{lemma}
Soit $f: V_{1} \to V_{2}$ un morphisme dans  $ {\cal M}_{o-tor}^{apf}(G)$. Alors $f(V_{1}^{+}) \subset V_{2}^{+}$.
\end{lemma}

\begin{proof} On factorise le morphisme $f: V_{1} \to V_{1}/Ker f \to f(V_{1}) \to V_{2}$ dans la cat\'egorie ab\'eliene  ${\cal M}_{o-tor}^{apf}(G)$.  On est ramen\'e au cas o\`u $f$ est surjectif ou une inclusion.
Si $f$ est surjective ou une inclusion, alors $f^{-1}(V_{2}^{+}) \subset  {\cal P}^{+}(V_{1})$  (\cite{SVig} Lemma 2.3  pour une surjection et Lemma 2.1 pour une inclusion) donc $f(V_{1}^{+} ) \subset V_{2}^{+}$. 
\end{proof}
L'application $V\mapsto V^{+}$ d\'efinit donc un foncteur de   ${\cal M}_{o-tor}^{apf}(G)$ vers   ${\cal M}_{o-tor}(B^{+})$.
\begin{remark}
 Les travaux de Hu  le sugg\'erent de se poser la question suivante : Existe-t-il un plongement  canonique de    ${\cal M}_{o-tor}^{apf}(G)$ dans  la cat\'egorie des diagrammes  ? Il n'est pas difficile de construire des foncteurs de    ${\cal M}_{o-tor}^{apf}(G)$ dans  la cat\'egorie des diagrammes. 
 Par exemple, remplacer ${\cal P}^{+}(V)$ par le sous-ensemble ${\cal P}_{I}^{+}(V)$ de ses \'el\'ements $I$-invariants. Cet ensemble a aussi un \'el\'ement minimal $V^{+}_{I}$ et $d_{V}:=V^{+}_{I}  \cap \w V^{+}_{I}$ est stable par $IZ \cap \w IZ$. On a donc le diagramme $$d_{V}\quad  \subset \quad <Kd_{V}> \ \quad $$Sa construction  est fonctorielle.
 \end{remark}

 Nous supposons dor\'enavant que $F$ est une extension finie de $\bf Q_{p}$. On a d\'efini au th\'eor\`eme \ref{3},
 un foncteur contravariant 
 $V\mapsto {\cal D}(V)$  de la cat\'egorie des repr\'esentations admissibles  de  pr\'esentation finie dans $ {\cal M}_{o-tor}(G)$  vers la cat\'egorie des $(\varphi, \Gamma)$-modules \'etales sur $O_{\cal E}$.
 On d\'eduit de la partie 2) du corollaire \ref{aut} que ${\cal D}(V)$ est un $O_{\cal E}$-module de type fini. On note ${\cal V}(V) \in   {\cal M}_{o}^{f}(\Gal _{\bf Q_{p}})$ la repr\'esentation galosienne associ\'ee par l'\'equivalence de cat\'egories de Fontaine.
 
 \begin{corollary}\label{cor5} Supposons que $F$ est une extension finie de $\bf Q_{p}$. Alors 
 $\cal V$ est  un foncteur contravariant de $ {\cal M}_{o-tor}^{apf}(G)$  vers $ {\cal M}_{o}^{f}(\Gal _{\bf Q_{p}})$.
 \end{corollary}

 Lorsque $F= \bf Q_{p}$, 
chaque repr\'esentation irr\'eductible admissible  $V\in {\cal M}_{o}(G)$ a une pr\'esentation finie (\cite{Serre80}; il y a plusieurs d\'emonstrations, la premi\`ere due \`a Colmez).
On d\'eduit de la proposition \ref{suite}, du lemme \ref{ext}, et du th\'eor\`eme \ref{lf} que 
${\cal M}_{o-tor}^{alf}(G) = {\cal M}_{o-tor}^{apf}(G)$. Le foncteur ${\cal D}(V)$ est exact car le probl\`eme pour l'exactitude vient seulement de la trace de $O_{F}$ \`a $\bf Z_{p}$.  
Ceci termine la preuve du th\'eor\`eme \ref{0} donn\'e dans l'introduction.


\begin{thebibliography}{}
 \bibitem{Breuil} Breuil Christophe, Diagrammes de Diamond et (phi,Gamma)-modules. A paraitre, Israel J. of Math.
 
 
 \bibitem{Colmez} Colmez Pierre, Repr\'esentations de $GL(2,{\bf Q_{p}})$ et (phi,Gamma)-modules. Preprint 2008, r\'evis\'e 2009.

 \bibitem{Hu}  Hu Yongquan, Diagrammes canoniques et repr\'esentations modulo $p$ de $GL_{2}(F)$.
 Preprint 2009.
 
 \bibitem{SVig}   Schneider P., Vigneras M.-F., A functor from smooth $o$-torsion representations to $(\varphi, \Gamma)$-modules. Preprint 2008.
 
\bibitem{Lang} Vigneras M.-F., Repr\'esentations $p$-adiques de torsion admissibles. Preprint 2006. R\'evis\'e 2009.



\bibitem{Serre80} Vigneras M.-F., A Criterion for Integral Structures and Coefficient Systems on the Tree of PGL(2, F). Pure and Applied Mathematics Quarterly Volume 4, Number 4 (Special Issue: In honor of Jean Pierre, Part 1 of 2) 1Ñ29, 2008.

\end{thebibliography}
\end{document}